\documentclass[11pt,a4paper,notitlepage]{article}%
\usepackage{amsfonts}
\usepackage{amsmath}
\usepackage{amssymb}
\usepackage{graphicx}%
\providecommand{\U}[1]{\protect\rule{.1in}{.1in}}
\newtheorem{theorem}{Theorem}
\newtheorem{acknowledgement}[theorem]{Acknowledgement}

\newtheorem{corollary}[theorem]{Corollary}

\newtheorem{definition}[theorem]{Definition}

\newtheorem{lemma}[theorem]{Lemma}
\newtheorem{notation}[theorem]{Notation}

\newtheorem{proposition}[theorem]{Proposition}
\newtheorem{remark}[theorem]{Remark}

\newenvironment{proof}[1][Proof]{\noindent\textbf{#1.} }{\ \rule{0.5em}{0.5em}}
\ifx\pdfoutput\relax\let\pdfoutput=\undefined\fi
\newcount\msipdfoutput
\ifx\pdfoutput\undefined\else
\ifcase\pdfoutput\else
\ifx\paperwidth\undefined\else
\ifdim\paperheight=0pt\relax\else\pdfpageheight\paperheight\fi
\ifdim\paperwidth=0pt\relax\else\pdfpagewidth\paperwidth\fi
\fi\fi\fi
\begin{document}

\title{A pseudo-differential calculus on non-standard symplectic space; spectral and
regularity results in modulation spaces}
\author{Nuno Costa Dias
\and Maurice de Gosson
\and Franz Luef
\and Jo\~{a}o Nuno Prata}
\maketitle

\begin{abstract}
The usual Weyl calculus is intimately associated with the choice of the
standard symplectic structure on $\mathbb{R}^{n}\oplus\mathbb{R}^{n}$. In this
paper we will show that the replacement of this structure by an arbitrary
symplectic structure leads to a pseudo-differential calculus of operators
acting on functions or distributions defined, not on $\mathbb{R}^{n}$ but
rather on $\mathbb{R}^{n}\oplus\mathbb{R}^{n}$. These operators are
intertwined with the standard Weyl pseudo-differential operators using an
infinite family of partial isometries of $L^{2}(\mathbb{R}^{n})\longrightarrow
L^{2}(\mathbb{R}^{2n})$ \ indexed by $\mathcal{S}(\mathbb{R}^{n})$. This
allows us obtain spectral and regularity results for our operators using
Shubin's symbol classes and Feichtinger's modulation spaces.

\end{abstract}

\section{Introduction}

Every traditional pseudo-differential calculus harks back in one way or
another to the physicists' early work on quantum mechanics. Following the
founding fathers of quantum mechanics one should associate to a symbol (or
\textquotedblleft observable\textquotedblright) defined on $\mathbb{R}%
^{2n}\equiv\mathbb{R}^{n}\oplus\mathbb{R}^{n}$ an operator obtained by
replacing the coordinates $x_{j}$ by the operator $\widehat{X}_{j}$ of
multiplication by $x_{j}$ and the dual variable $\xi_{j}$ by the operator
$\widehat{\Xi}_{j}=-i\partial_{x_{j}}$. The ordering problem (what is the
operator associated with $\xi_{j}x_{j}=x_{j}\xi_{j}$?) was solved in a
satisfactory way by Weyl \cite{Weyl}: one associates to the symbol $a$ the
operator $\widehat{A}=\operatorname*{Op}^{w}(a)$ with kernel formally defined
by%
\begin{equation}
K(x,y)=\left(  \tfrac{1}{2\pi}\right)  ^{n}\int_{\mathbb{R}^{n}}%
e^{i(x-y)\cdot\xi}a(\tfrac{1}{2}(x+y),\xi)d\xi. \label{kxy}%
\end{equation}
The Weyl correspondence $a\overset{\text{Weyl}}{\longleftrightarrow
}\widehat{A}$ plays a somewhat privileged role among the other possible
choices $a\overset{\tau}{\longleftrightarrow}A_{\tau}$ corresponding to the
kernels%
\begin{equation}
K_{\tau}(x,y)=\left(  \tfrac{1}{2\pi}\right)  ^{n}\int_{\mathbb{R}^{n}%
}e^{i(x-y)\cdot\xi}a(\tau x+(1-\tau)y),\xi)d\xi\label{ktauxy}%
\end{equation}
with $\tau\in\mathbb{R}$. This is due mainly to two reasons: first of all, the
choice (\ref{kxy}) ensures us that to real symbols correspond (formally)
self-adjoint operators; secondly, among all possible choices (\ref{ktauxy})
the Weyl correspondence $a\overset{\text{Weyl}}{\longleftrightarrow
}\widehat{A}$ is the only one which has the symplectic covariance property
$a\circ S\overset{\text{Weyl}}{\longleftrightarrow}\widehat{S}^{-1}%
\widehat{A}\widehat{S}$ where $\widehat{S}\in\operatorname*{Mp}(2n,\sigma)$
has projection $S\in\operatorname*{Sp}(2n,\sigma)$ ($\operatorname*{Sp}%
(2n,\sigma)$ and $\widehat{S}\in\operatorname*{Mp}(2n,\sigma)$ are the
symplectic and metaplectic groups, respectively). It turns out that the Weyl
correspondence is intimately related to the standard symplectic structure
$\sigma(z,z^{\prime})=\xi\cdot x^{\prime}-\xi^{\prime}\cdot x$ on
$\mathbb{R}^{n}\oplus\mathbb{R}^{n}$ or, equivalently, to the commutation
relations%
\begin{equation}
\lbrack\widehat{X}_{j},\widehat{X}_{k}]=[\widehat{\Xi}_{j},\widehat{\Xi}%
_{k}]=0\text{ \ , \ }[\widehat{X}_{j},\widehat{\Xi}_{k}]=i\delta_{jk}
\label{CCR1}%
\end{equation}
satisfied by the elementary Weyl operators $\widehat{X}_{j},\widehat{\Xi}_{k}%
$. Setting $\widehat{Z}_{\alpha}=\widehat{X}_{\alpha}$ if $1\leq\alpha\leq n$
and $\widehat{Z}_{\alpha}=\widehat{\Xi}_{\alpha-n}$ if $n+1\leq\alpha\leq2n$
these relations can be rewritten
\begin{equation}
\lbrack\widehat{Z}_{\alpha},\widehat{Z}_{\beta}]=ij_{\alpha\beta}\text{ \ for
\ }1\leq\alpha,\beta\leq2n \label{CCR1bis}%
\end{equation}
where%
\[
J=(j_{\alpha\beta})_{1\leq\alpha,\beta\leq2n}=%
\begin{pmatrix}
0 & I\\
-I & 0
\end{pmatrix}
\]
is the matrix of the symplectic form $\sigma$. Here $I,0$ denote the $n\times
n$ identity and zero matrices, respectively.

We now make the two following essential observations:

\begin{itemize}
\item There are many operators satisfying the commutation relations
(\ref{CCR1})--(\ref{CCR1bis}). For instance, they are preserved if one
replaces $\widehat{X}_{j}$ and $\widehat{\Xi}_{j}$ with the operators\
\begin{equation}
\widetilde{X}_{j}=x_{j}+\tfrac{1}{2}i\partial_{\xi_{j}}\text{ \ ,
\ }\widetilde{\Xi}_{j}=\xi_{j}-\tfrac{1}{2}i\partial_{x_{j}} \label{ccr3}%
\end{equation}
(these are the \textquotedblleft Bopp shifts\textquotedblright\ \cite{Bopp}
familiar from the physical literature). Notice that $\widetilde{X}_{j}$ and
$\widetilde{\Xi}_{j}$ act not on functions defined on $\mathbb{R}^{n}$ but
rather on functions defined on $\mathbb{R}^{n}\oplus\mathbb{R}^{n}$. Indeed,
in recent papers de Gosson \cite{CPDE}, de Gosson and Luef \cite{GOLU1,GOLU2},
Dias et al. \cite{digoprlu1} it has been shown that the operators
$\widetilde{X}_{j}$ and $\widetilde{\Xi}_{j}$ can be used to reformulate the
Moyal product familiar from deformation quantization \cite{BFFLS1,BFFLS2} in
terms of a phase-space pseudo-differential calculus, which also intervenes in
the study of certain magnetic operators (\textquotedblleft Landau
calculus\textquotedblright\ \cite{CPDE});

\item The second observation takes us to the subject of this
paper. The choice of the standard symplectic structure, associated
with the commutation relations (\ref{CCR1bis}), is to a great
extent arbitrary. So one could wonder what happens if we replace
the matrix $J=(j_{\alpha\beta})_{1\leq\alpha ,\beta\leq2n}$ with
some other non-degenerate skew-symmetric matrix $\Omega$. This
question is not only academic: the study of non-commutative field
theories and their connections with quantum gravity
\cite{dipra1,badipr,babedipr,babedipr1,Douglas,Szabo} has led
physicists to consider
more general commutation relations of the type%
\begin{equation}
\lbrack\widetilde{Z}_{\alpha},\widetilde{Z}_{\beta}]=i\omega_{\alpha\beta
}\text{ \ for \ }1\leq\alpha,\beta\leq2n \label{CCR2}%
\end{equation}
where the numbers $\omega_{\alpha\beta}$ are defined by
\begin{equation}
\Omega=(\omega_{\alpha\beta})_{1\leq\alpha,\beta\leq2n}=%
\begin{pmatrix}
\Theta & I\\
-I & N
\end{pmatrix}
\label{omega}%
\end{equation}
where $\Theta=(\theta_{\alpha\beta})_{1\leq\alpha,\beta\leq n}$ and
$N=(\eta_{\alpha\beta})_{1\leq\alpha,\beta\leq n}$ are antisymmetric matrices
(see \cite{babedipr,bracz,cahakola}). The commutation relations (\ref{CCR2})
are satisfied by the operators
\begin{align}
\widetilde{X}_{j}  &  =x_{j}+\tfrac{1}{2}i\partial_{\xi_{j}}+\tfrac{1}{2}%
i\sum\nolimits_{k}\theta_{jk}\partial_{x_{k}}\label{ccr31}\\
\widetilde{\Xi}_{j}  &  =\xi_{j}-\tfrac{1}{2}i\partial_{x_{j}}+\tfrac{1}%
{2}i\sum\nolimits_{k}\eta_{jk}\partial_{\xi_{k}} \label{ccr32}%
\end{align}
which reduce to the \textquotedblleft Bopp shifts\textquotedblright%
\ (\ref{ccr3}) when $\Omega=J$. The relation of these operators with a
deformation quantization has been made explicit in Dias et al.
\cite{digoprlu1}.
\end{itemize}

Writing formulas (\ref{ccr31})--(\ref{ccr32}) in compact form as
\begin{equation}
\widetilde{Z}=z+\tfrac{1}{2}i\Omega\partial_{z} \label{cc4}%
\end{equation}
this suggests that one should be able to give a sense to pseudo-differential
operators formally written as
\begin{equation}
\widetilde{A}_{\omega}=a(\widetilde{Z})=a(z+\tfrac{1}{2}i\Omega\partial_{z}).
\label{atild1}%
\end{equation}
We set out in this paper to justify formula (\ref{atild1}); more generally we
define a pseudo-differential calculus arising from the choice of an arbitrary
symplectic form $\omega$ with constant coefficients on $\mathbb{R}^{n}%
\oplus\mathbb{R}^{n}$ associated to an antisymmetric matrix $\Omega\in
GL(2n;\mathbb{R})$ by the formula
\[
\omega(z,z^{\prime})=z\cdot\Omega^{-1}z^{\prime}.
\]
This symplectic form obviously coincides with the standard symplectic form
$\sigma$ when $\Omega=J=%
\begin{pmatrix}
0 & I\\
-I & 0
\end{pmatrix}
$. The consideration of such operators $\widetilde{A}_{\omega}$ leads to a
class of Weyl operators with symbols defined on $\mathbb{R}^{2n}%
\oplus\mathbb{R}^{2n}$.

In this article we will show that:

\begin{itemize}
\item The formal definition (\ref{atild1}) of the operators $\widetilde{A}%
_{\omega}$ and their Weyl symbols can be made rigorous;

\item The operators $\widetilde{A}_{\omega}$ are intertwined with the usual
Weyl operators $\widehat{A}$ using a family of partial isometries
$u\longmapsto W_{f,\phi}u$ of $L^{2}(\mathbb{R}^{n})$ in $L^{2}(\mathbb{R}%
^{2n})$ parametrized by $\phi\in\mathcal{S}(\mathbb{R}^{n})$;

\item The spectral properties of the operators $\widetilde{A}_{\omega}$ can be
recovered from those of $\widehat{A}$ using these intertwining relations; in
particular the consideration of Shubin's classes of globally hypoelliptic
symbols will allow us to state a very precise result when $\widehat{A}$ is
formally self-adjoint.
\end{itemize}

Our results show that the study of the physicist's \textquotedblleft
non-commutative quantum mechanics\textquotedblright\ can be reduced to the
study of a particular Weyl calculus with symbols defined on a double phase space.

We want to mention that the connections between symbol classes and
non-commutative harmonic analysis have also been explored (from a different
point of view) by Unterberger \cite{Unterberger1} and Unterberger and Upmeier
\cite{Unterberger2}; it would perhaps be interesting to analyze their results
from the point of view of the methods and tools introduced in the present paper.

\begin{notation}
The generic point of $\mathbb{R}^{n}\oplus\mathbb{R}^{n}\equiv\mathbb{R}^{2n}$
is denoted by $z=(x,\xi)$ and that of $\mathbb{R}^{2n}\oplus\mathbb{R}%
^{2n}\equiv\mathbb{R}^{4n}$ by $(z,\zeta)$. The standard symplectic form
$\sigma$ on $\mathbb{R}^{2n}$ is defined by $\sigma(z,z^{\prime})=\xi\cdot
x^{\prime}-\xi^{\prime}\cdot x$ and the corresponding symplectic group is
denoted $\operatorname*{Sp}(2n,\sigma)$. Given an arbitrary symplectic form
$\omega$ on $\mathbb{R}^{n}\oplus\mathbb{R}^{n}$ we denote by
$\operatorname*{Sp}(2n,\omega)$ the corresponding symplectic group.
\end{notation}

\begin{notation}
Functions (or distributions) on $\mathbb{R}^{n}$ are denoted by small Latin or
Greek letters $u,v,\phi$,... while those defined on $\mathbb{R}^{2n}$ by
capitals $U,V,\Phi,...$ We denote by $\mathcal{S}(\mathbb{R}^{n})$ the
Schwartz space of rapidly decreasing functions on $\mathbb{R}^{n}$; its dual
$\mathcal{S}^{\prime}(\mathbb{R}^{n})$ is the space of tempered distributions.
The scalar product of two functions $u,v\in L^{2}(\mathbb{R}^{n})$ is denoted
by $(u|v)$ and that of $U,V\in L^{2}(\mathbb{R}^{2n})$ by $(\!(U \vert V)\!)$.
The corresponding norms are written $||u||$ and $|||U|||$.
\end{notation}

\section{Phase Space Weyl Operators}

Let us begin by giving a short review of the main definitions and properties
from standard Weyl calculus as exposed (with fluctuating notation) in for
instance \cite{Birk,Hor2,Shubin,Stein,Wong}; this will allow us to list some
useful formulas we will need in the forthcoming sections.

\subsection{Standard Weyl calculus\label{subsecone}}

Given a function $a\in\mathcal{S}(\mathbb{R}^{2n})$ the Weyl operator
$\widehat{A}$ with symbol $a$ is defined by:%
\begin{equation}
\widehat{A}u(x)=\left(  \tfrac{1}{2\pi}\right)  ^{n}\iint\nolimits_{\mathbb{R}%
^{2n}}e^{i(x-y)\cdot\xi}a(\tfrac{1}{2}(x+y),\xi)u(y)dyd\xi\label{ahato}%
\end{equation}
for $u\in\mathcal{S}(\mathbb{R}^{n})$. This definition makes sense for more
general symbols $a$ provided that the integral interpreted in some
\textquotedblleft reasonable way\textquotedblright\ (oscillatory integral, for
instance) when $a$ is in a suitable symbol class, for instance the
H\"{o}rmander classes $S_{\rho,\delta}^{m}$, or the global Shubin spaces
$H\Gamma_{\rho}^{m_{1},m_{0}}$. A better definition is, no doubt, the operator
integral
\begin{equation}
\widehat{A}=\left(  \tfrac{1}{2\pi}\right)  ^{n}\int_{\mathbb{R}^{2n}%
}F_{\sigma}a(z)\widehat{T}(z)dz \label{ahat}%
\end{equation}
because it immediately makes sense for arbitrary symbols $a\in\mathcal{S}%
^{\prime}(\mathbb{R}^{2n})$; here $F_{\sigma}$ is the symplectic Fourier
transform:
\begin{equation}
F_{\sigma}a(z)=\left(  \tfrac{1}{2\pi}\right)  ^{n}\int_{\mathbb{R}^{2n}%
}e^{-i\sigma(z,z^{\prime})}a(z^{\prime})dz^{\prime} \label{sft}%
\end{equation}
$\widehat{T}(z_{0})$ is the Heisenberg--Weyl operator $\mathcal{S}^{\prime
}(\mathbb{R}^{n})\longrightarrow\mathcal{S}^{\prime}(\mathbb{R}^{n})$ formally
defined by%
\begin{equation}
\widehat{T}(z_{0})=e^{-i\sigma(\widehat{Z},z_{0})}\text{ \ \textit{with}
\ }\widehat{Z}=(x,-i\partial_{x}); \label{hwo}%
\end{equation}
the action of $\widehat{T}(z_{0})$ on $u\in\mathcal{S}(\mathbb{R}^{n})$ is
given by the explicit formula
\begin{equation}
\widehat{T}(z_{0})u(x)=e^{i(\xi_{0}\cdot x-\frac{1}{2}\xi_{0}\cdot x_{0}%
)}u(x-x_{0}) \label{hw}%
\end{equation}
if $z_{0}=(x_{0},\xi_{0})$. We note that $F_{\sigma}$ is an involution which
extends into an involutive automorphism $\mathcal{S}^{\prime}(\mathbb{R}%
^{2n})\longrightarrow\mathcal{S}^{\prime}(\mathbb{R}^{2n})$. The Weyl
correspondence, written $a\overset{\text{Weyl}}{\longleftrightarrow
}\widehat{A}$ or $\widehat{A}\overset{\text{Weyl}}{\longleftrightarrow}a$,
between an element $a\in\mathcal{S}^{\prime}(\mathbb{R}^{2n})$ and the Weyl
operator it defines is bijective; in fact the Weyl transformation is
one-to-one from $\mathcal{S}^{\prime}(\mathbb{R}^{2n})$ onto the space
$\mathcal{L}(\mathcal{S}(\mathbb{R}^{n}),\mathcal{S}^{\prime}(\mathbb{R}%
^{2n}))$ of continuous maps $\mathcal{S}(\mathbb{R}^{n})\longrightarrow
\mathcal{S}^{\prime}(\mathbb{R}^{n})$ (see e.g. Maillard \cite{Maillard}, Wong
\cite{Wong}). This can be proved using Schwartz's kernel theorem and the fact
that the Weyl symbol $a$ of the operator $\widehat{A}$ is related to the
distributional kernel of that operator by the partial Fourier transform with
respect to the $y$ variable
\begin{equation}
a(x,\xi)=\int_{\mathbb{R}^{n}}e^{-i\xi\cdot y}K(x+\tfrac{1}{2}y,x-\tfrac{1}%
{2}y)dy \label{ak}%
\end{equation}
where
$K\in\mathcal{S}^{\prime}(\mathbb{R}^{n}\times\mathbb{R}^{n})$ and
the Fourier transform is defined in the usual distributional
sense. Conversely (cf. formula (\ref{ahato})) the kernel $K$ is
expressed in terms of
the symbol $a$ by the inverse Fourier transform%
\[
K(x,y)=\left(  \tfrac{1}{2\pi}\right)  ^{n}\int_{\mathbb{R}^{n}}e^{i\xi
\cdot(x-y)}a(\tfrac{1}{2}(x+y),\xi)d\xi.
\]

Assuming that the product $\widehat{A}\widehat{B}$ exists (which is the case
for instance if $\widehat{B}:\mathcal{S}(\mathbb{R}^{n})\longrightarrow
\mathcal{S}(\mathbb{R}^{n})$) the Weyl symbol $c$ of $\widehat{C}%
=\widehat{A}\widehat{B}$ and its symplectic Fourier transform $F_{\sigma}c$
are given by the formulas
\begin{gather}
c(z)=\left(  \tfrac{1}{4\pi}\right)  ^{2n}\iint\nolimits_{\mathbb{R}%
^{2n}\times\mathbb{R}^{2n}}e^{\frac{i}{2}\sigma(u,v)}a(z+\tfrac{1}%
{2}u)b(z-\tfrac{1}{2}v)dudv\label{compo1}\\
F_{\sigma}c(z)=\left(  \tfrac{1}{2\pi}\right)  ^{n}\int_{\mathbb{R}^{2n}%
}e^{\frac{i}{2}\sigma(z,z^{\prime})}F_{\sigma}a(z-z^{\prime})F_{\sigma
}b(z^{\prime})dz^{\prime}. \label{compo2}%
\end{gather}
The first formula is often written $c=a\#b$ and $a\#b$ is called the
\textquotedblleft twisted product\textquotedblright\ or \textquotedblleft
Moyal product\textquotedblright\ (see e.g. \cite{Wong}).

Two important properties of the Weyl correspondence already mentioned in the
Introduction are the following:

\begin{proposition}
Let $\widehat{A}\overset{\text{Weyl}}{\longleftrightarrow}a$:

(i) The operator $\widehat{A}$ is formally self-adjoint if and only the symbol
$a$ is real; more generally the symbol of the formal adjoint of an operator
with Weyl symbol $a$ is its complex conjugate $\overline{a}$;

(ii) Let $\widehat{S}\in\operatorname*{Mp}(2n,\sigma)$. We have $\widehat{S}%
^{-1}\widehat{A}\widehat{S}\overset{\text{Weyl}}{\longleftrightarrow}a\circ S$.
\end{proposition}

Here $\operatorname*{Mp}(2n,\sigma)$ is the metaplectic group, that is the
unitary representation of the double cover of $\operatorname*{Sp}(2n,\sigma)$.
To every $S\in\operatorname*{Sp}(2n,\sigma)$ thus corresponds, via the natural
projection $\pi:\operatorname*{Mp}(2n,\sigma)\longrightarrow\operatorname*{Sp}%
(2n,\sigma)$, two operators $\pm\widehat{S}\in\operatorname*{Mp}(2n,\sigma)$.
We note that property (ii) is \emph{characteristic} of the Weyl
pseudo-differential calculus (see Stein \cite{Stein}, Wong \cite{Wong}). We
notice that Unterberger and Upmeier \cite{Unterberger2} have studied similar
covariance formula for more general operators (pseudo-differential operators
of Fuchs type) which occur in the study of boundary problems with edges or corners.

A related well-known object is the cross-Wigner transform $W(u,v)$ of
$u,v\in\mathcal{S}(\mathbb{R}^{n})$; it is defined by%
\begin{equation}
W(u,v)(z)=\left(  \tfrac{1}{2\pi}\right)  ^{n}\int_{\mathbb{R}^{n}}%
e^{-i\xi\cdot y}u(x+\tfrac{1}{2}y)\overline{v(x-\tfrac{1}{2}y)}dy \label{wifi}%
\end{equation}
(it is thus, up to a constant, the Weyl symbol of the operator with kernel
$u\otimes\overline{v}$). We note, for further use, that $W(u,v)$ can
alternatively be defined by the formula%
\begin{equation}
W(u,v)(z)=\pi^{-n}\langle\widehat{T}_{\text{GR}}(z)u,\overline{v}%
\rangle\label{wigroyer}%
\end{equation}
where $\widehat{T}_{\text{GR}}(z)$ is the Grossmann--Royer operator:%
\begin{equation}
\widehat{T}_{\text{GR}}(z_{0})u(x)=e^{2i\xi_{0}\cdot(x-x_{0})}u(2x_{0}-x).
\label{tgr}%
\end{equation}
Formula (\ref{wigroyer}) allows us to define $W(u,v)$ when $u\in
\mathcal{S}^{\prime}(\mathbb{R}^{n})$ and $v\in\mathcal{S}(\mathbb{R}^{n})$;
one can actually extend the mapping $(u,v)\longrightarrow W(u,v)$ into a
continuous mapping $\mathcal{S}^{\prime}(\mathbb{R}^{n})\times\mathcal{S}%
^{\prime}(\mathbb{R}^{n})\longrightarrow\mathcal{S}^{\prime}(\mathbb{R}^{2n}%
)$. The cross-Wigner transform enjoys the following symplectic-covariance
property: if $S\in\operatorname*{Sp}(2n,\sigma)$ then%
\begin{equation}
W(u,v)(S^{-1}z)=W(\widehat{S}u,\widehat{S}v)(z) \label{sycov}%
\end{equation}
where $\widehat{S}\in\operatorname*{Mp}(2n,\sigma)$ has projection $S$. Let
$u,v\in\mathcal{S}(\mathbb{R}^{n})$. The following important property is
sometimes taken as the definition of the Weyl operator $\widehat{A}$:
\begin{equation}
(\widehat{A}u|v)=\int_{\mathbb{R}^{2n}}a(z)W(u,v)(z)dz=\langle a,W(u,v)\rangle
. \label{ww}%
\end{equation}
Also note that the cross-Wigner transform satisfies the Moyal identity
\begin{equation}
(\!(W(u,v) \vert W(u^{\prime},v^{\prime}))\!)=\left(  \tfrac{1}{2\pi}\right)
^{n}(u|u^{\prime})\overline{(v|v^{\prime})}. \label{Moyal}%
\end{equation}

The following formula describes the action of the Heisenberg--Weyl operators:%
\begin{equation}
W(\widehat{T}(z_{0})u,\widehat{T}(z_{1})v)(z)=e^{i[-\sigma(z,z_{0}%
-z_{1})-\frac{1}{2}\sigma(z_{0},z_{1})]}W(u,v)(z-\langle z\rangle);
\label{wt1}%
\end{equation}
where $\langle z\rangle=\frac{1}{2}(z_{0}+z_{1})$; the particular case%
\begin{equation}
W(\widehat{T}(z_{0})u,v)(z)=e^{-i\sigma(z,z_{0})}W(u,v)(z-\tfrac{1}{2}z_{0})
\label{wt2}%
\end{equation}
will be used in our study of intertwining operators.

\subsection{Definition of the operators $\protect\widetilde{A}_{\omega}$}

In what follows $\Omega$ denotes an arbitrary (real) invertible antisymmetric
$2n\times2n$ \ matrix. The formula
\begin{equation}
\omega(z,z^{\prime})=z\cdot\Omega^{-1}z^{\prime}=-\Omega^{-1}z\cdot z^{\prime}
\label{omf}%
\end{equation}
defines a symplectic form on $\mathbb{R}^{2n}$ which coincides with the
standard symplectic form $\sigma$ when $\Omega=J$.

Let us introduce the following variant of the symplectic Fourier transform:

\begin{definition}
\label{noncomsympfour} For $a\in\mathcal{S}(\mathbb{R}^{2n})$ we set
\begin{equation}
F_{\omega}a(z)=\left(  \tfrac{1}{2\pi}\right)  ^{n}|\det\Omega|^{-1/2}%
\int_{\mathbb{R}^{2n}}e^{-i\omega(z,z^{\prime})}a(z^{\prime})dz^{\prime}.
\label{aoum}%
\end{equation}

\end{definition}

Obviously $F_{\omega}$ is a continuous automorphism of $\mathcal{S}%
(\mathbb{R}^{2n})$. Moreover:

\begin{lemma}
\label{lemnoncomsympfour}The automorphism $F_{\omega}$ extends
into a unitary automorphism of $L^{2}(\mathbb{R}^{2n})$ and into a continuous automorphism of $\mathcal{S}%
^{\prime}(\mathbb{R}^{2n})$. Moreover, $F_{\omega}$ is related to
the usual unitary Fourier transform $F$ on $\mathbb{R}^{2n}$ by the formula%
\begin{equation}
Fa(z)=|\det\Omega|^{1/2}F_{\omega}a(-\Omega z). \label{foufou}%
\end{equation}
In particular $F_{\omega}$ is involutive, that is
\begin{equation}
F_{\omega}F_{\omega}a=a. \label{finv}%
\end{equation}

\end{lemma}

\begin{remark}
Notice that we are using the normalization of the Fourier transform according
to the rule $(2\pi)^{-\mbox{dimension}/2}$. Since we are working in the
phase-space ($\mbox{dimension}=2n$), we have a factor $(2\pi)^{-n}$ rather
than the usual factor $(2\pi)^{-n/2}$.
\end{remark}

\begin{proof}
From $\omega(-\Omega z,z^{\prime})=z\cdot z^{\prime}$, we
immediately recover (\ref{foufou}). From (\ref{foufou}) and the
unitarity of the Fourier transform, we have in
$L^{2}(\mathbb{R}^{2n})$:
\begin{equation}%
\begin{array}
[c]{c}%
|||a|||=|||Fa|||=|\det\Omega|^{1/2}\left(  \int_{\mathbb{R}^{2n}}|F_{\omega
}a(-\Omega z)|^{2}dz\right)  ^{\frac{1}{2}}=\\
\\
=\left(  \int_{\mathbb{R}^{2n}}|F_{\omega}a(z^{\prime})|^{2}dz^{\prime
}\right)  ^{\frac{1}{2}}=|||F_{\omega}a|||,
\end{array}
\label{foufou2}%
\end{equation}
where we performed the substitution $z^{\prime}=-\Omega z$. Consequently,
$F_{\omega}$ extends into a unitary automorphism of $L^{2}(\mathbb{R}^{2n})$.
The symplectic Fourier transform $F_{\omega}$ also extends into a continuous
automorphism of $\mathcal{S}^{\prime}(\mathbb{R}^{2n})$ in the usual way by
defining $F_{\omega}a$ for $a\in\mathcal{S}^{\prime}(\mathbb{R}^{2n})$ by the
formula $\langle F_{\omega}a,b\rangle=\langle a,F_{\omega}b\rangle$ for all
$b\in\mathcal{S}(\mathbb{R}^{2n})$ (or, alternatively, by using the relation
(\ref{foufou}) above). Note that when $\Omega=J$ we have $F_{\omega}%
=F_{\sigma}$ (the ordinary symplectic Fourier transform) since $\det J=1$.
Using formula (\ref{foufou}) the symplectic Fourier transform $F_{\omega}$ can
thus be written:
\begin{equation}
F_{\omega}=U_{\Omega}IF\label{foufou3}%
\end{equation}
where $U_{\Omega}$ and $I$ are the transformations defined by
\begin{equation}
(U_{\Omega}a)(z)=|\det\Omega|^{1/2}a(\Omega^{-1}z),\hspace{1cm}%
(Ia)(z)=a(-z)\label{foufou4}%
\end{equation}
for which trivially:
\begin{equation}
(\!( U_{\Omega}a\vert U_{\Omega}b)\!)=(\!( a\vert
b)\!),\hspace{1cm}(\!(
Ia\vert Ib)\!)=(\!( a \vert b)\!)\label{foufou5}%
\end{equation}
for all $a,b\in L^{2}(\mathbb{R}^{2n})$. From (\ref{foufou5}) and the Parseval
identity, it follows that for all $a,b\in L^{2}(\mathbb{R}^{2n})$:
\begin{equation}
(\!( F_{\omega}F_{\omega}a\vert b)\!)=(\!( F_{\omega}a\vert F_{\omega}%
b)\!)=(\!( U_{\Omega}IFa\vert U_{\Omega}IFb)\!)=(\!( a\vert b)\!)
\label{foufou6}%
\end{equation}
which proves (\ref{finv}).
\end{proof}

In the sequel we will also need the operators
\[
\widetilde{T}_{\omega}(z_{0}):\mathcal{S}^{\prime}(\mathbb{R}^{2n}%
)\longrightarrow\mathcal{S}^{\prime}(\mathbb{R}^{2n})
\]
defined by the formula
\begin{equation}
\widetilde{T}_{\omega}(z_{0})U(z)=e^{-i\omega(z,z_{0})}U(z-\tfrac{1}{2}z_{0}).
\label{hwbis}%
\end{equation}
These operators satisfy the same commutation relations as the usual
Heisenberg--Weyl operators $\widehat{T}(z_{0})$ when $\omega=\sigma$. In fact,
a straightforward computation shows that%
\begin{align}
\widetilde{T}_{\omega}(z_{0}+z_{1})  &  =e^{-\tfrac{i}{2}\omega(z_{0},z_{1}%
)}\widetilde{T}_{\omega}(z_{0})\widetilde{T}_{\omega}(z_{1})\label{tom1}\\
\widetilde{T}_{\omega}(z_{0})\widetilde{T}_{\omega}(z_{1})  &  =e^{i\omega
(z_{0},z_{1})}\widetilde{T}_{\omega}(z_{1})\widetilde{T}_{\omega}(z_{0}).
\label{tom2}%
\end{align}
Let us justify the introduction of the operators $\widetilde{T}_{\omega}%
(z_{0})$ with an informal discussion; after all it is not obvious at this
stage why they should allow us to implement the \textquotedblleft
quantization\textquotedblright\ (\ref{cc4})--(\ref{atild1})! Recall
\cite{Birk} that the introduction of the usual Heisenberg--Weyl operator
$\widehat{T}(z_{0})=e^{-i\sigma(\widehat{z},z_{0})}$ can be motivated as
follows: consider the translation Hamiltonian $H_{z_{0}}(z)=\sigma(z,z_{0})$;
the operator with this Weyl symbol is $\widehat{H}_{z_{0}}(z)=\sigma
(\widehat{Z},z_{0})$ and the solution of the corresponding Schr\"{o}dinger
equation%
\[
i\frac{\partial}{\partial t}u=\widehat{H}_{z_{0}}u\text{ \ , \ }%
u(x,0)=u_{0}(x)
\]
is formally given by $u(x,t)=e^{-it\sigma(\widehat{Z},z_{0})}u_{0}(x)$; a
direct calculation shows that we have the explicit formula
\[
u(x,t)=e^{- it\sigma( \widehat{Z},z_{0})}u_{0}(x)=e^{i(t\xi_{0}\cdot
x-\frac{1}{2}t^{2}\xi_{0}\cdot x_{0})}u_{0}(x-tx_{0})
\]
hence $\widehat{T}(z_{0})u(x,0)=u(x,1)$. To define the operators
$\widetilde{T}_{\omega}(z_{0})$ one proceeds exactly in the same way:
replacing the Hamiltonian operator $\widehat{H}_{z_{0}}(z)=\sigma
(\widehat{Z},z_{0})$ with%
\[
\widetilde{H}_{z_{0}}(z)=\omega(\widetilde{Z},z_{0})=\omega(z+\tfrac{1}%
{2}i\Omega\partial_{z},z_{0})
\]
we are led to the \textquotedblleft phase space Schr\"{o}dinger
equation\textquotedblright%
\[
i\frac{\partial}{\partial t}U=\omega(\widetilde{Z},z_{0})U\text{ \ ,
\ }U(z,0)=U_{0}(z)
\]
whose solution is
\[
U(z,t)=e^{- it\omega(\widetilde{Z},z_{0})}U_{0}(z)=e^{-it\omega(z,z_{0})}%
U_{0}(z-\tfrac{1}{2}tz_{0}).
\]
We thus have%
\[
U(z,1)=\widetilde{T}_{\omega}(z_{0})U_{0}(z)=e^{- i\omega(\widetilde{Z}%
,z_{0})}U_{0}(z).
\]

Let us now define the operators $\widetilde{A}_{\omega}$. Comparing with the
definition (\ref{ahat}) of the usual Weyl operators these considerations
suggest that we define $\widetilde{A}_{\omega}=a(\widetilde{Z})$ by the
formula
\begin{equation}
\widetilde{A}_{\omega}U=\left(  \tfrac{1}{2\pi}\right)  ^{n}|\det
\Omega|^{-1/2}\int_{\mathbb{R}^{2n}}F_{\omega}a(z)\widetilde{T}_{\omega
}(z)Udz. \label{atilde}%
\end{equation}
This \textquotedblleft guess\textquotedblright\ is justified by the following
result which identifies the Weyl symbol of the operator $\widetilde{A}%
_{\omega}$ defined by the formula above:

\begin{proposition}
\label{propone}Let $a\in\mathcal{S}^{\prime}(\mathbb{R}^{2n})$ and
$U\in\mathcal{S}(\mathbb{R}^{2n})$. The operator $\widetilde{A}_{\omega
}:\mathcal{S}(\mathbb{R}^{2n})\longrightarrow\mathcal{S}^{\prime}%
(\mathbb{R}^{2n})$ defined by%
\begin{equation}
\widetilde{A}_{\omega}U=\left(  \tfrac{1}{2\pi}\right)  ^{n}|\det
\Omega|^{-1/2}\langle F_{\omega}a(\cdot),\widetilde{T}_{\omega}(\cdot
)U\rangle\label{acroch}%
\end{equation}
that is, formally, by (\ref{atilde}) is continuous and its Weyl symbol
$\widetilde{a}_{\omega}$ is given by the formula%
\begin{equation}
\widetilde{a}_{\omega}(z,\zeta)=a\left(  z-\tfrac{1}{2}\Omega\zeta\right)
\label{azj}%
\end{equation}
and we have $\widetilde{a}_{\omega}\in\mathcal{S}^{\prime}(\mathbb{R}%
^{2n}\oplus\mathbb{R}^{2n})$. When $a=1$ the operator $\widetilde{A}_{\omega}$
is the identity on $\mathcal{S}(\mathbb{R}^{2n})$.
\end{proposition}

\begin{proof}
Since $\widetilde{T}_{\omega}(z)U\in\mathcal{S}(\mathbb{R}^{2n})$ for every
$z$ and $F_{\omega}a\in\mathcal{S}^{\prime}(\mathbb{R}^{2n})$ the operator
$\widetilde{A}_{\omega}$ is well-defined. We have, setting $u=z-\frac{1}%
{2}z_{0}$,%
\begin{align*}
\widetilde{A}_{\omega}U(z)  &  =\left(  \tfrac{1}{2\pi}\right)  ^{n}%
|\det\Omega|^{-1/2}\int_{\mathbb{R}^{2n}}F_{\omega}a(z_{0})\widetilde{T}%
_{\omega}(z_{0})U(z)dz_{0}\\
&  =\left(  \tfrac{1}{2\pi}\right)  ^{n}|\det\Omega|^{-1/2}\int_{\mathbb{R}%
^{2n}}F_{\omega}a(z_{0})e^{-i\omega(z,z_{0})}U(z-\tfrac{1}{2}z_{0})dz_{0}\\
&  =\left(  \tfrac{2}{\pi}\right)  ^{n}|\det\Omega|^{-1/2}\int_{\mathbb{R}%
^{2n}}F_{\omega}a[2(z-u)]e^{2i\omega(z,u)}U(u)du
\end{align*}
hence the kernel of $\widetilde{A}_{\omega}$ is given by the formula%
\[
K(z,u)=\left(  \tfrac{2}{\pi}\right)  ^{n}|\det\Omega|^{-1/2}F_{\omega
}a[2(z-u)]e^{2i\omega(z,u)}.
\]
It follows from formula (\ref{ak}) that the symbol $\widetilde{a}_{\omega}$ is
given by%
\begin{align*}
\widetilde{a}_{\omega}(z,\zeta)  &  =\int_{\mathbb{R}^{2n}}e^{-i\zeta
\cdot\zeta^{\prime}}K(z+\tfrac{1}{2}\zeta^{\prime},z-\tfrac{1}{2}\zeta
^{\prime})d\zeta^{\prime}\\
&  =\left(  \tfrac{2}{\pi}\right)  ^{n}|\det\Omega|^{-1/2}\int_{\mathbb{R}%
^{2n}}e^{-i\zeta\cdot\zeta^{\prime}}F_{\omega}a(2\zeta^{\prime})e^{-2i\omega
(z,\zeta^{\prime})}d\zeta^{\prime}%
\end{align*}
that is, using the obvious relation
\[
\zeta\cdot\zeta^{\prime}+2\omega(z,\zeta^{\prime})=\omega(2z-\Omega\zeta
,\zeta^{\prime})
\]
together with the change of variables $z^{\prime}=2\zeta^{\prime}$,
\begin{align*}
\widetilde{a}_{\omega}(z,\zeta)  &  =\left(  \tfrac{2}{\pi}\right)  ^{n}%
|\det\Omega|^{-1/2}\int_{\mathbb{R}^{2n}}e^{-i\omega(2z-\Omega\zeta
,\zeta^{\prime})}F_{\omega}a(2\zeta^{\prime})d\zeta^{\prime}\\
&  =\left(  \tfrac{1}{2\pi}\right)  ^{n}|\det\Omega|^{-1/2}\int_{\mathbb{R}%
^{2n}}e^{-i\omega(z-\frac{1}{2}\Omega\zeta,z^{\prime})}F_{\omega}a(z^{\prime
})d z^{\prime}.
\end{align*}
Formula (\ref{azj}) immediately follows using the Fourier inversion formula
(\ref{finv}). That $\widetilde{A}_{\omega}=I$ when $a=1$ immediately follows
from the fact that $F_{\omega}a=(2\pi)^{n} | \det\Omega|^{1/2} \delta$ where
$\delta$ is the Dirac measure on $\mathbb{R}^{2n}$. The continuity statement
follows from the fact that $\widetilde{A}_{\omega}$ is a Weyl operator.
\end{proof}

Two immediate consequences of this result are:

\begin{corollary}
The operators $\widetilde{A}_{\omega}$ have the following properties:

(i) The operator $\widetilde{A}_{\omega}$ defined by (\ref{atilde}) is
formally self-adjoint if and only if $a$ is real;

(ii) The formal adjoint $\widetilde{A}_{\omega}^{\ast}$ of $\widetilde{A}%
_{\omega}$ is obtained by replacing $a$ with its complex conjugate
$\overline{a}$;

(iii) The symbol $\widetilde{c}$ of $\widetilde{C}_{\omega}=\widetilde{A}%
_{\omega}\widetilde{B}_{\omega}$ is given by $\widetilde{c}_{\omega}%
(z,\zeta)=c\left(  z-\tfrac{1}{2}\Omega\zeta\right)  $ where $c=a\#b$ is the
Weyl symbol of the operator $\widehat{C}=\widehat{A}\widehat{B}$.
\end{corollary}

\begin{proof}
(i) The property is obvious since $\widetilde{A}_{\omega}$ is formally
self-adjoint if and only if its Weyl symbol $\widetilde{a}_{\omega}$ is real,
that is if and only if $a$ itself is real. (ii) Similarly, the Weyl symbol of
$\widetilde{A}_{\omega}^{\ast}$ is the function
\[
(z,\zeta)\longmapsto\overline{a\left(  z-\tfrac{1}{2}\Omega\zeta\right)  }.
\]
(iii) The property is an immediate consequence of the definition of
$\widetilde{C}_{\omega}$ since $a\#b\overset{\text{Weyl}}{\longleftrightarrow
}\widehat{A}\widehat{B}$.
\end{proof}

\subsection{Symplectic transformation properties}

Let $\omega$ be the symplectic form (\ref{omf}) on $\mathbb{R}^{n}%
\oplus\mathbb{R}^{n}$. The symplectic spaces $(\mathbb{R}^{n}\oplus
\mathbb{R}^{n},\omega)$ and $(\mathbb{R}^{n}\oplus\mathbb{R}^{n},\sigma)$ are
linearly symplectomorphic. That is, there exists a linear automorphism $f$ of
$\mathbb{R}^{2n}$ such that $f^{\ast}\omega=\sigma$ that is
\begin{equation}
\omega(fz,fz^{\prime})=\sigma(z,z^{\prime}) \label{omsig}%
\end{equation}
for all $(z,z^{\prime})\in\mathbb{R}^{2n}\times\mathbb{R}^{2n}$ (this can be
viewed as a linear version of Darboux's theorem). The proof is
straightforward: choose a symplectic basis $\mathcal{B}$ of $(\mathbb{R}%
^{n}\oplus\mathbb{R}^{n},\omega)$ and a symplectic basis $\mathcal{B}^{\prime
}$ of $(\mathbb{R}^{n}\oplus\mathbb{R}^{n},\sigma)$. Then any linear
automorphism $f$ of $\mathbb{R}^{n}\oplus\mathbb{R}^{n}$ such that
$f(\mathcal{B}^{\prime})=\mathcal{B}$ satisfies (\ref{omsig}). Identifying the
automorphism $f$ with its matrix in the canonical basis, the relation
(\ref{omsig}) is equivalent to the matrix equality%
\begin{equation}
\Omega=fJf^{T}. \label{oj}%
\end{equation}
Such a symplectomorphism $f:(\mathbb{R}^{2n},\sigma)\longrightarrow
(\mathbb{R}^{2n},\omega)$ is by no means unique; we can in fact replace it by
any automorphism $f^{\prime}=fS_{\sigma}$ where $S_{\sigma}\in
\operatorname*{Sp}(2n,\sigma)$; note however that the determinant is an
invariant because we have
\[
\det f^{\prime}=\det f\det S_{\sigma}=\det f
\]
since $\det S_{\sigma}=1$. The symplectic groups $\operatorname*{Sp}%
(\mathbb{R}^{2n},\omega)$ and $\operatorname*{Sp}(\mathbb{R}^{2n},\sigma)$ are
canonically isomorphic.

We are going to see that the study of the operators $\widetilde{A}_{\omega}$
is easily reduced to the case where $\omega=\sigma$, the standard symplectic
form on $\mathbb{R}^{2n}$. This result is closely related to the symplectic
covariance of Weyl operators under metaplectic conjugation as we will see below.

For $f$ a linear automorphism of $\mathbb{R}^{2n}$ we define the operator
\[
M_{f}:\mathcal{S}^{\prime}(\mathbb{R}^{2n})\longrightarrow\mathcal{S}^{\prime
}(\mathbb{R}^{2n})
\]
by the formula
\begin{equation}
M_{f}U(z)=\sqrt{|\det f|}U(fz)\text{.} \label{pushpush}%
\end{equation}
Clearly $M_{f}$ is unitary: we have $|||M_{f}U|||=|||U|||$ for all $U\in
L^{2}(\mathbb{R}^{2n})$.

\begin{notation}
When $\Omega=J$ we write $\widetilde{T}(z_{0})=\widetilde{T}_{\sigma}(z_{0})$
and $\widetilde{A}=\widetilde{A}_{\sigma}$.
\end{notation}

\begin{proposition}
\label{proptwo}Let $f:(\mathbb{R}^{2n},\sigma)\longrightarrow(\mathbb{R}%
^{2n},\omega)$ be a linear symplectomorphism.

(i) We have the conjugation formulas%
\begin{gather}
M_{f}\widetilde{T}_{\omega}(z_{0})=\widetilde{T}(f^{-1}z_{0})M_{f}\text{ ,
}M_{f}F_{\omega}=F_{\sigma}M_{f}\text{ \ \ }\label{equ}\\
M_{f}\widetilde{A}_{\omega}=\widetilde{A^{\prime}}M_{f}\text{\ with
\ }a^{\prime}(z)=a(fz). \label{equa}%
\end{gather}

(ii) When $f$ is replaced by an automorphism $f^{\prime}=fS_{\sigma}$ with
$S_{\sigma}\in\operatorname*{Sp}(2n,\sigma)$ then $\widetilde{A^{\prime}}$ is
replaced by the operator%
\begin{equation}
\widetilde{A^{\prime\prime}}=M_{S_{\sigma}}\widetilde{A^{\prime}}M_{S_{\sigma
}}^{-1} \label{conjug}%
\end{equation}
where $M_{S_{\sigma}}U(z)=U(S_{\sigma}z)$.
\end{proposition}

\begin{proof}
(i) Since $\omega(fz,z_{0})=\sigma(z,f^{-1}z_{0})$ we have for all
$U\in\mathcal{S}^{\prime}(\mathbb{R}^{2n})$,%
\begin{align*}
M_{f}\left[  \widetilde{T}_{\omega}(z_{0})U\right]  (z)  &  =\sqrt{|\det
f|}e^{-i\omega(fz,z_{0})}U(fz-\tfrac{1}{2}z_{0})\\
&  =\sqrt{|\det f|}e^{-i\sigma(z,f^{-1}z_{0})}U(f(z-\tfrac{1}{2}f^{-1}%
z_{0}))\\
&  =e^{-i\sigma(z,f^{-1}z_{0})}M_{f}U(z-\tfrac{1}{2}f^{-1}z_{0})\\
&  =\widetilde{T}(f^{-1}z_{0})M_{f}U(z)
\end{align*}
which is equivalent to the first equality (\ref{equ}). We have likewise for
$a\in\mathcal{S}^{\prime}(\mathbb{R}^{2n})$%
\begin{align*}
M_{f}F_{\omega}a(z)  &  =\sqrt{|\det f|}F_{\omega}a(fz)\\
&  =\left(  \tfrac{1}{2\pi}\right)  ^{n}|\det\Omega|^{-1/2}\sqrt{|\det f|}%
\int_{\mathbb{R}^{2n}}e^{-i\omega(fz,z^{\prime})}a(z^{\prime})dz^{\prime}\\
&  =\left(  \tfrac{1}{2\pi}\right)  ^{n}|\det\Omega|^{-1/2}\sqrt{|\det f|}%
\int_{\mathbb{R}^{2n}}e^{-i\sigma(z,f^{-1}z^{\prime})}a(z^{\prime})dz^{\prime
}\\
&  =\left(  \tfrac{1}{2\pi}\right)  ^{n}|\det\Omega|^{-1/2}|\det
f|\int_{\mathbb{R}^{2n}}e^{-i\sigma(z,z^{\prime\prime})}M_{f}a(z^{\prime
\prime})dz^{\prime\prime}%
\end{align*}
hence the second equality (\ref{equ}) because%
\begin{equation}
|\det\Omega|^{-1/2}|\det f|=1 \label{a}%
\end{equation}
in view of the equality (\ref{oj}). To prove that $M_{f}\widetilde{A}_{\omega
}=\widetilde{A^{\prime}}M_{f}$ it suffices to use the relations (\ref{equ})
together with definition (\ref{atilde}) of $\widetilde{A}_{\omega}$:%
\begin{align*}
M_{f}\widetilde{A}_{\omega}  &  =\left(  \tfrac{1}{2\pi}\right)  ^{n}%
|\det\Omega|^{-1/2}\int_{\mathbb{R}^{2n}}F_{\omega}a(z)M_{f}\widetilde{T}%
_{\omega}(z)dz\\
&  =\left(  \tfrac{1}{2\pi}\right)  ^{n}|\det\Omega|^{-1/2}\int_{\mathbb{R}%
^{2n}}F_{\omega}a(z)\widetilde{T}(f^{-1}z)M_{f}dz;
\end{align*}
performing the change of variables $z\longmapsto fz$ we get, using again
(\ref{a}), and noting that $|\det f|^{-1/2}M_{f}a(z)=a(fz)$,
\begin{align*}
M_{f}\widetilde{A}_{\omega}  &  =\left(  \tfrac{1}{2\pi}\right)  ^{n}%
|\det\Omega|^{-1/2}|\det f|\int_{\mathbb{R}^{2n}}F_{\omega}a(fz)\widetilde{T}%
(z)M_{f}dz\\
&  =\left(  \tfrac{1}{2\pi}\right)  ^{n}\int_{\mathbb{R}^{2n}}F_{\omega
}a(fz)\widetilde{T}(z)M_{f}dz\\
&  =\left(  \tfrac{1}{2\pi}\right)  ^{n}|\det f|^{-1/2}\int_{\mathbb{R}^{2n}%
}M_{f}F_{\omega}a(z)\widetilde{T}(z)M_{f}dz\\
&  =\left(  \tfrac{1}{2\pi}\right)  ^{n}|\det f|^{-1/2}\int_{\mathbb{R}^{2n}%
}F_{\sigma}M_{f}a(z)\widetilde{T}(z)M_{f}dz\\
&  =\left(  \tfrac{1}{2\pi}\right)  ^{n}\int_{\mathbb{R}^{2n}}F_{\sigma
}(a\circ f)(z)\widetilde{T}(z)M_{f}dz\\
&  =\widetilde{A^{\prime}}M_{f}.
\end{align*}
(ii) To prove formula (\ref{conjug}) it suffices to note that%
\begin{align*}
M_{f^{\prime}}\widetilde{A}_{\omega}  &  =(M_{f^{\prime}}M_{f}^{-1}%
)M_{f}\widetilde{A}_{\omega}\\
&  =M_{S_{\sigma}}(\widetilde{A^{\prime}}M_{f})\\
&  =(M_{S_{\sigma}}\widetilde{A^{\prime}}M_{S_{\sigma}}^{-1})M_{S_{\sigma}%
}M_{f}\\
&  =(M_{S_{\sigma}}\widetilde{A^{\prime}}M_{S_{\sigma}}^{-1})M_{f^{\prime}}.
\end{align*}
That we have $M_{S_{\sigma}}U(z)=U(S_{\sigma}z)$ is clear since $\det
S_{\sigma}=1$.
\end{proof}

We note that formula (\ref{conjug}) can be interpreted in terms of the
symplectic covariance property of Weyl calculus. To see this, let us equip the
double phase space $\mathbb{R}^{2n}\oplus\mathbb{R}^{2n}$ with the symplectic
structure $\sigma^{\oplus}=\sigma\oplus\sigma$. In view of formula (\ref{azj})
with $\Omega=J$ the Weyl symbols of operators $\widetilde{A^{\prime\prime}}$
and $\widetilde{A^{\prime}}$ are, respectively%
\[
\widetilde{a^{\prime}}(z,\zeta)=a\left[  f(z-\tfrac{1}{2}J\zeta)\right]
\text{ \ , \ }\widetilde{a^{^{\prime\prime}}}(z,\zeta)=a\left[  f^{\prime
}(z-\tfrac{1}{2}J\zeta)\right]
\]
and hence, using the identities $f^{-1}f^{\prime}=S_{\sigma}\in
\operatorname*{Sp}(2n,\sigma)$ and $S_{\sigma}J=J(S_{\sigma}^{T})^{-1}$,
\[
\widetilde{a^{^{\prime\prime}}}(z,\zeta)=a^{\prime}\left[  S_{\sigma}%
(z-\tfrac{1}{2}J(S_{\sigma}^{T})^{-1}\zeta)\right]  =\widetilde{a^{\prime}%
}(S_{\sigma}z,(S_{\sigma}^{T})^{-1}\zeta).
\]
Let now $m_{S_{\sigma}}$ be the automorphism of $\mathbb{R}^{2n}%
\oplus\mathbb{R}^{2n}$ defined by
\[
m_{S_{\sigma}}(z,\zeta)=(S_{\sigma}^{-1}z,S_{\sigma}^{T}\zeta);
\]
formula (\ref{conjug}) can thus be restated as%
\begin{equation}
\widetilde{A^{\prime\prime}}=M_{S_{\sigma}}\widetilde{A^{\prime}}M_{S_{\sigma
}}^{-1}\text{ \ with \ }a^{^{\prime\prime}}=a^{\prime}\circ m_{S_{\sigma}%
}^{-1}. \label{conjugbis}%
\end{equation}
Recall now (see for instance \cite{Birk}, Chapter 7) that each automorphism
$f$ of $\mathbb{R}^{2n}$ induces an element $m_{f}$ of $\operatorname*{Sp}%
(4n,\sigma^{\oplus})$ defined by $m_{f}(z,\zeta)=(f^{-1}z,f^{T}\zeta)$ and
that $m_{f}$ is the projection of the metaplectic operator $\widehat{M}_{f}%
\in\operatorname*{Mp}(\mathbb{R}^{2n}\oplus\mathbb{R}^{2n},\sigma^{\oplus})$
(with $\sigma^{\oplus}=\sigma\oplus\sigma$) defined by (\ref{pushpush}).
Formula (\ref{conjugbis}) thus reflects the symplectic covariance property of
Weyl calculus mentioned in Subsection \ref{subsecone}.

We finally note that if we equip $\mathbb{R}^{2n}\oplus\mathbb{R}^{2n}$ with
the symplectic form $\omega^{\oplus}=\omega\oplus\omega$, the
symplectomorphism $f:(\mathbb{R}^{2n},\sigma)\longrightarrow(\mathbb{R}%
^{2n},\omega)$ induces a natural symplectomorphism%
\[
f\oplus f:(\mathbb{R}^{2n}\oplus\mathbb{R}^{2n},\sigma^{\oplus}%
)\longrightarrow(\mathbb{R}^{2n}\oplus\mathbb{R}^{2n},\omega^{\oplus}).
\]

\section{The Intertwining Property}

In this section we show that the operators $\widetilde{A}_{\omega}$ can be
intertwined with the standard Weyl operator $\widehat{A}$ using an infinite
family of partial isometries $(W_{f,\phi})_{\phi\in\mathcal{S}(\mathbb{R}%
^{n})}$ of $L^{2}(\mathbb{R}^{n})$ (depending on $\Omega$) into $L^{2}%
(\mathbb{R}^{2n})$. Each $W_{f,\phi}$ maps isomorphically $L^{2}%
(\mathbb{R}^{n})$ onto a closed subspace $\mathcal{H}_{\phi}$ of
$L^{2}(\mathbb{R}^{2n})$.

\subsection{The partial isometries $W_{f,\phi}$}

Let $\phi\in\mathcal{S}(\mathbb{R}^{n})$ be such that $||\phi||=1$; $\phi$
will be hereafter called a \emph{window}. In \cite{GOLU1} two of us have
studied the linear mapping $W_{\phi}:\mathcal{S}(\mathbb{R}^{n}%
)\longrightarrow\mathcal{S}(\mathbb{R}^{2n})$ defined by the formula
\begin{equation}
W_{\phi}u=(2\pi)^{n/2}W(u,\phi) \label{wpf}%
\end{equation}
where $W(u,\phi)$ is the cross-Wigner distribution (\ref{wifi}). Notice that
\begin{equation}
W_{\phi}u(z)=\left(  \tfrac{2}{\pi}\right)  ^{n/2}(\widehat{T}_{\text{GR}%
}(z)u|\phi) \label{ufigr}%
\end{equation}
where $\widehat{T}_{\text{GR}}(z)$ is the Grossmann--Royer transform
(\ref{tgr}).

\begin{proposition}
Let $\phi\in\mathcal{S}(\mathbb{R}^{n})$ be a window.

(i) The mapping $W_{\phi}:\mathcal{S}(\mathbb{R}^{n})\longrightarrow
\mathcal{S}(\mathbb{R}^{2n})$ extends into a mapping%
\[
W_{\phi}:\mathcal{S}^{\prime}(\mathbb{R}^{n})\longrightarrow\mathcal{S}%
^{\prime}(\mathbb{R}^{2n})
\]
whose restriction to $L^{2}(\mathbb{R}^{n})$ is an isometry onto a closed
subspace $\mathcal{H}_{\phi}$ of $L^{2}(\mathbb{R}^{2n})$.

(ii) The inverse of $W_{\phi}$ is given by the formula $u=W_{\phi}^{-1}U$ with%
\begin{equation}
u(x)=\left(  \tfrac{2}{\pi}\right)  ^{n/2}\int_{\mathbb{R}^{2n}}%
U(z_{0})\widehat{T}_{\text{GR}}(z_{0})\phi(x)dz_{0} \label{psitgrufi}%
\end{equation}
and the adjoint $W_{\phi}^{\ast}$ of $W_{\phi}$ is given by the formula%
\begin{equation}
W_{\phi}^{\ast}U=\left(  \tfrac{2}{\pi}\right)  ^{n/2}\int_{\mathbb{R}^{2n}%
}U(z_{0})\widehat{T}_{\text{GR}}(z_{0})\phi(x) dz_{0}. \label{psitadj}%
\end{equation}

(iii) The operator $P_{\phi}=W_{\phi}W_{\phi}^{\ast}$ is the orthogonal
projection of $L^{2}(\mathbb{R}^{2n})$ onto the Hilbert space $\mathcal{H}%
_{\phi}$.
\end{proposition}

\begin{proof}
(i) In view of Moyal's identity (\ref{Moyal}) the operator $W_{\phi}$ extends
into an isometry of $L^{2}(\mathbb{R}^{n})$ onto a subspace $\mathcal{H}%
_{\phi}$ of $L^{2}(\mathbb{R}^{2n})$:%
\[
(\!(W_{\phi}u|W_{\phi}u^{\prime})\!)=(u|u^{\prime}).
\]
The subspace $\mathcal{H}_{\phi}$ is closed, being homeomorphic to
$L^{2}(\mathbb{R}^{n})$. (ii) The inversion formula (\ref{psitgrufi}) is
verified by a direct calculation: let us set%
\[
w(x)=\left(  \tfrac{2}{\pi}\right)  ^{n/2}\int_{\mathbb{R}^{2n}}%
U(z_{0})\widehat{T}_{\text{GR}}(z_{0})\phi(x)dz_{0}%
\]
and choose an arbitrary function $v\in\mathcal{S}(\mathbb{R}^{n})$. We have%
\begin{align*}
(w|v)  &  =\left(  \tfrac{2}{\pi}\right)  ^{n/2}\int_{\mathbb{R}^{2n}}%
U(z_{0})(\widehat{T}_{\text{GR}}(z_{0})\phi|v)dz_{0}\\
&  =(2\pi)^{n/2}\int_{\mathbb{R}^{2n}}U(z_{0})\overline{W(v,\phi)}%
(z_{0})dz_{0}\\
&  =\int_{\mathbb{R}^{2n}}W_{\phi}u(z_{0})\overline{W_{\phi}v(z_{0})}dz_{0}\\
&  =(u|v)
\end{align*}
hence $w=u$ which proves (\ref{psitgrufi}); formula (\ref{psitadj}) for the
adjoint follows since $W_{\phi}^{\ast}W_{\phi}$ is the identity on
$L^{2}(\mathbb{R}^{n})$. (iii) We have $P_{\phi}=P_{\phi}^{\ast}$ and
$P_{\phi}P_{\phi}^{\ast}=P_{\phi}$ hence $P_{\phi}$ is an orthogonal
projection. Since $W_{\phi}^{\ast}W_{\phi}$ is the identity on $L^{2}%
(\mathbb{R}^{n})$ the range of $W_{\phi}^{\ast}$ is $L^{2}(\mathbb{R}^{n})$
and that of $P_{\phi}$ is therefore precisely $\mathcal{H}_{\phi}$.
\end{proof}

In \cite{GOLU1} it was shown that the partial isometries $W_{\phi}$ can be
used to intertwine the operators $\widetilde{A}=\widetilde{A}_{\sigma}$ with
symbol $\widetilde{a}$ with the usual Weyl operators with symbol $a$; we
reproduce the proof for convenience:

\begin{proposition}
Let $\widetilde{T}(z_{0})=\widetilde{T}_{\sigma}(z_{0})$. We have the
following intertwining properties:
\begin{equation}
W_{\phi}\widehat{T}(z_{0})=\widetilde{T}(z_{0})W_{\phi}\text{\ and \ }W_{\phi
}^{\ast}\widetilde{T}(z_{0})=\widehat{T}(z_{0})W_{\phi}^{\ast} \label{inter1}%
\end{equation}%
\begin{equation}
\widetilde{A}W_{\phi}=W_{\phi}\widehat{A}\text{ \ and \ }W_{\phi}^{\ast
}\widetilde{A}=\widehat{A}W_{\phi}^{\ast}\text{.} \label{fund}%
\end{equation}

\end{proposition}

\begin{proof}
Formula (\ref{inter1}) immediately follows from the shift property
(\ref{wt2}). On the other hand we have
\[
W_{\phi}\widehat{A}u=\left(  \tfrac{1}{2\pi}\right)  ^{n}\int_{\mathbb{R}%
^{2n}}F_{\sigma}a(z_{0})W_{\phi}[\widehat{T}(z_{0})u]dz_{0}%
\]
and hence, in view of (\ref{inter1}),%
\[
W_{\phi}\widehat{A}u=\left(  \tfrac{1}{2\pi}\right)  ^{n}\int_{\mathbb{R}%
^{2n}}F_{\sigma}a(z_{0})[\widetilde{T}(z_{0})W_{\phi}u]dz_{0}%
\]
which is the first equality (\ref{fund}). To prove the second equality
(\ref{fund}) it suffices to apply the first to $W_{\phi}^{\ast}\widetilde{A}%
=(\widetilde{A}^{\ast}W_{\phi})^{\ast}$.
\end{proof}

Let us generalize this result to the case of an arbitrary operator
$\widetilde{A}_{\omega}$.

\begin{proposition}
\label{aprime}Let $\omega$ be a symplectic form (\ref{omf}) on $\mathbb{R}%
^{2n}$ and $f$ a linear automorphism such that $f^{\ast}\omega=\sigma$. The
mappings $W_{f,\phi}:\mathcal{S}(\mathbb{R}^{n})\longrightarrow S(\mathbb{R}%
^{2n})$ defined by the formula:%
\begin{equation}
W_{f,\phi}=M_{f}^{-1}W_{\phi} \label{ws}%
\end{equation}
are partial isometries $L^{2}(\mathbb{R}^{n})\longrightarrow L^{2}%
(\mathbb{R}^{2n})$, in fact isometries on a closed subspace $\mathcal{H}%
_{f,\phi}$ of $L^{2}(\mathbb{R}^{2n})$, and we have%
\begin{equation}
\widetilde{A}_{\omega}W_{f,\phi}=W_{f,\phi}\widehat{A^{\prime}}\text{ \ and
\ }W_{f,\phi}^{\ast}\widetilde{A}_{\omega}=\widehat{A^{\prime}}W_{f,\phi
}^{\ast} \label{sfund}%
\end{equation}
where $\widehat{A^{\prime}}\overset{\text{Weyl}}{\longleftrightarrow}a\circ f$.
\end{proposition}

\begin{proof}
We have, using the first formula (\ref{fund}) and (\ref{equa}),
\begin{align*}
\widetilde{A}_{\omega}W_{f,\phi}  &  =M_{f}^{-1}\widetilde{A^{\prime}}%
M_{f}(M_{f}^{-1}W_{\phi})\\
&  =M_{f}^{-1}(\widetilde{A^{\prime}}W_{\phi})\\
&  =M_{f}^{-1}W_{\phi}\widehat{A^{\prime}}\\
&  =W_{f,\phi}\widehat{A^{\prime}};
\end{align*}
the equality $W_{f,\phi}^{\ast}\widetilde{A}_{\omega}=\widehat{A^{\prime}%
}W_{f,\phi}^{\ast}$ is proven in a similar way. That $W_{f,\phi}$ is a partial
isometry is obvious since $W_{\phi}$ is a a partial isometry and $M_{f}$ is unitary.
\end{proof}

Let us make explicit the change of the mapping $f$:

\begin{proposition}
Let $f$ and $f^{\prime}$ be linear automorphisms of $\mathbb{R}^{2n}$ such
that $f^{\ast}\omega=f^{\prime\ast}\omega=\sigma$. We have%
\begin{equation}
W_{f^{\prime},\phi}u=W_{f,\widehat{S}_{\sigma}\phi}(\widehat{S}_{\sigma}u)
\label{usfi}%
\end{equation}
where $\widehat{S}_{\sigma}\in\operatorname*{Mp}(2n,\sigma)$ is such that
$\pi(\widehat{S}_{\sigma})=f^{-1}f^{\prime}$.
\end{proposition}

\begin{proof}
The relation $f^{\ast}\omega=f^{\prime\ast}\omega=\sigma$ implies that
$S_{\sigma}=f^{-1}f^{\prime}\in\operatorname*{Sp}(2n,\sigma)$. We have
$M_{f^{\prime}}=M_{fS_{\sigma}}=M_{S_{\sigma}}M_{f}$ and hence%
\[
W_{f^{\prime},\phi}=M_{f^{\prime}}^{-1}W_{\phi}=M_{f}^{-1}M_{S_{\sigma}}%
^{-1}W_{\phi}.
\]
Now, taking into account definition (\ref{wpf}) of $W_{\phi}$ in terms of the
cross-Wigner transform and the fact that $\det S_{\sigma}=1$ we have, using
the symplectic covariance property (\ref{sycov}),
\begin{align*}
M_{S_{\sigma}}^{-1}W_{\phi}u(z)  &  =(2\pi)^{n/2}W(u,\phi)(S_{\sigma}^{-1}z)\\
&  =(2\pi)^{n/2}W(\widehat{S}_{\sigma}u,\widehat{S}_{\sigma}\phi)(z)\\
&  =W_{\widehat{S}_{\sigma}\phi}(\widehat{S}_{\sigma}u)(z)
\end{align*}
hence formula (\ref{usfi}).
\end{proof}

We remark that the union of the ranges of the partial isometries $W_{\phi}$
viewed as mappings defined on $\mathcal{S}^{\prime}(\mathbb{R}^{n})$ is in a
sense a rather small subset of $\mathcal{S}^{\prime}(\mathbb{R}^{2n})$ even
when $\phi$ runs over all of $\mathcal{S}^{\prime}(\mathbb{R}^{n})$; this is a
consequence of Hardy's theorem on the concentration of a function and its
Fourier transform (de Gosson and Luef \cite{GOLULETT,GOLUAHP}), and is related
to a topological formulation of the uncertainty principle (de Gosson
\cite{FP}). we will discuss these facts somewhat more in detail at the end of
the article.

\subsection{Action of $W_{f,\phi}$ on orthonormal bases}

Let us prove the following important result that shows that orthonormal bases
of $L^{2}(\mathbb{R}^{n})$ can be used to generate orthonormal bases of
$L^{2}(\mathbb{R}^{2n})$ using the mappings $W_{f,\phi}$:

\begin{proposition}
\label{propob}Let $(\phi_{j})_{j}$ be a complete family of vectors in
$L^{2}(\mathbb{R}^{n})$.

(i) The family $(\Phi_{j,k})_{j,k}$ with $\Phi_{j,k}=W_{f,\phi_{j}}\phi_{k}$
is complete in $L^{2}(\mathbb{R}^{2n})$.

(ii) If $(\phi_{j})_{j}$ is an orthonormal basis of $L^{2}(\mathbb{R}^{n})$
then $(\Phi_{j,k})_{j,k}$ is an orthonormal basis of $L^{2}(\mathbb{R}^{2n})$.
\end{proposition}

\begin{proof}
We first note that (ii) follows from (i) since $W_{f,\phi}$ is an isometry of
$L^{2}(\mathbb{R}^{n})$ onto its range $\mathcal{H}_{f,\phi}$ in
$L^{2}(\mathbb{R}^{2n})$. Let us show that if $U\in L^{2}(\mathbb{R}^{2n})$ is
orthogonal to the family $(\Phi_{j,k})_{j,k}$ (and hence to all the spaces
$\mathcal{H}_{f,\phi_{j}}$) then $U=0$. Since by definition $W_{f,\phi}%
=M_{f}^{-1}W_{\phi}$ and the image of a complete system of vectors by
$M_{f}^{-1}$ is also complete, it is sufficient to assume that $W_{f,\phi
}=W_{\phi}$.\ Suppose now that we have $(\!(U \vert\Phi_{j,k} )\!)=0$ for all
indices $j,k$. Since
\[
(\!(U \vert\Phi_{j,k} )\!)=(\!(U \vert W_{\phi_{j}}\phi_{k} )\!) =
(\!(W_{\phi_{j}}^{\ast}U \vert\phi_{k} )\!)
\]
it follows that $W_{\phi_{j}}^{\ast}U=0$ for all $j$ since $(\phi_{j})_{j}$ is
a basis; using the anti-linearity of $W_{\phi}$ in $\phi$ we have in fact
$W_{\phi}^{\ast}U=0$ for all $\phi\in L^{2}(\mathbb{R}^{n})$. Let us show that
this property implies that we must have $U=0$. Recall that the adjoint of the
wavepacket transform $W_{\phi}^{\ast}$ is given by
\[
W_{\phi}^{\ast}U=\left(  \tfrac{2}{\pi}\right)  ^{n/2}\int_{\mathbb{R}^{2n}%
}U(z_{0})\widehat{T}_{\text{GR}}(z_{0})\phi dz_{0}%
\]
where $\widehat{T}_{\text{GR}}(z_{0})$ is the Grossmann--Royer operator (see
formula (\ref{psitadj}) above). Let now $u$ be an arbitrary element of
$\mathcal{S}(\mathbb{R}^{n})$; we have, using definition (\ref{wigroyer}) of
the cross-Wigner transform,%
\begin{align*}
(W_{\phi}^{\ast}U|u)  &  =\left(  \tfrac{2}{\pi}\right)  ^{n/2}\int%
_{\mathbb{R}^{2n}}U(z)(\widehat{T}_{\text{GR}}(z)\phi|u)dz\\
&  =\left(  2\pi\right)  ^{n/2}\int_{\mathbb{R}^{2n}}U(z)W(\phi,u)(z)dz.
\end{align*}
Let us now view $(2\pi)^{n/2}U\in L^{2}(\mathbb{R}^{2n})$ as the Weyl symbol
of an operator $\widehat{A}_{U}$. In view of formula (\ref{ww}) we have%
\[
\left(  2\pi\right)  ^{n/2}\int_{\mathbb{R}^{2n}}U(z)W(\phi
,u)(z)dz=(\widehat{A}_{U}\phi|u)
\]
and the condition $W_{\phi}^{\ast}U=0$ for all $\phi\in\mathcal{S}%
(\mathbb{R}^{n})$ is thus equivalent to the condition $(\widehat{A}_{U}%
\phi|u)=0$ for all $\phi,u\in\mathcal{S}(\mathbb{R}^{n})$. It follows that
$\widehat{A}_{U}\phi=0$ for all $\phi$ and hence $\widehat{A}_{U}=0$. Since
the Weyl correspondence is one-to-one we must have $U=0$ as claimed.
\end{proof}

We remark that the argument in the proof above in fact allows to show that,
more generally, given two orthonormal bases $(\phi_{j})_{j}$ and $(\psi
_{j})_{j}$ of $L^{2}(\mathbb{R}^{n})$ the vectors $\Phi_{j,k}=W_{f,\phi_{j}%
}\psi_{k}$ form an orthonormal basis of $L^{2}(\mathbb{R}^{2n})$.

\section{Spectral Properties of the Operators $\protect\widetilde{A}_{\omega}%
$}

Particularly useful symbol classes for the study of the spectral properties
are the \textquotedblleft global\textquotedblright\ symbol classes
$H\Gamma_{\rho}^{m_{1},m_{0}}(\mathbb{R}^{2n})$ introduced in Shubin
\cite{Shubin}; also see Buzano et al. \cite{buniro10}.

\subsection{The Shubin symbol classes $H\Gamma_{\rho}^{m_{1},m_{0}}$}

Let $m_{0},m_{1}\in\mathbb{R}$ and $0<\rho\leq1$. Introducing the multi-index
notation $\alpha=(\alpha_{1},...,\alpha_{2n})\in\mathbb{N}^{2n}$,
$|\alpha|=\alpha_{1}+\cdot\cdot\cdot+\alpha_{2n}$, $\ $and $\partial
_{z}^{\alpha}=\partial_{x_{1}}^{\alpha_{1}}\cdot\cdot\cdot\partial_{x_{n}%
}^{\alpha_{n}}\partial_{y_{1}}^{\alpha_{n+1}}\cdot\cdot\cdot\partial_{y_{n}%
}^{\alpha_{2n}}$, we have by definition $a\in H\Gamma_{\rho}^{m_{1},m_{0}%
}(\mathbb{R}^{2n})$ if:

\begin{itemize}
\item \textit{We have} $a\in C^{\infty}(\mathbb{R}^{2n})$;

\item \textit{There exist constants} $R,C_{0},C_{1}\geq0$ \textit{and, for
every} $\alpha\in\mathbb{N}^{2n}$, $|\alpha|\neq0$, \textit{a constant}
$C_{\alpha}\geq0$ \textit{such that for} $|z|\geq R$ \textit{the following
estimates hold:}
\begin{equation}
C_{0}|z|^{m_{0}}\leq|a(z)|\leq C_{1}|z|^{m_{1}}\text{ \ , \ }|\partial
_{z}^{\alpha}a(z)|\leq C_{\alpha}|a(z)||z|^{-\rho|\alpha|}. \label{shu}%
\end{equation}

\end{itemize}

The first condition (\ref{shu}) is an ellipticity condition; observe that
$H\Gamma_{\rho}^{m_{1},m_{0}}(\mathbb{R}^{2n})$ is not a vector space.

A simple but typical example is the following: the function $a$ defined by
$a(z)=\frac{1}{2}|z|^{2}$ is in $H\Gamma_{1}^{2,2}(\mathbb{R}^{2n})$, the same
applies, more generally to $a(z)=\frac{1}{2}Mz\cdot z$ when $M$ is a real
positive definite matrix.

The interest of these symbol classes comes from the following result (Shubin
\cite{Shubin}, Chapter 4):

\begin{proposition}
\label{propeigen}Let $a\in H\Gamma_{\rho}^{m_{1},m_{0}}(\mathbb{R}^{2n})$ be
real, and $m_{0}>0$. Then the formally self-adjoint operator $\widehat{A}$
with Weyl symbol $a$ has the following properties: (i) $\widehat{A}$ is
essentially self-adjoint and has discrete spectrum in $L^{2}(\mathbb{R}^{n})$;
(ii) There exists an orthonormal basis of eigenfunctions $\phi_{j}%
\in\mathcal{S}(\mathbb{R}^{n})$ ($j=1,2,...$) with eigenvalues $\lambda_{j}%
\in\mathbb{R}$ such that $\lim_{j\rightarrow\infty}|\lambda_{j}|=\infty$.
\end{proposition}

We observe that in the Proposition above there exists a basis of
eigenfunctions belonging to $\mathcal{S}(\mathbb{R}^{n})$; this property
follows from the global hypoellipticity of operators with Weyl symbol in
$H\Gamma_{\rho}^{m_{1},m_{0}}(\mathbb{R}^{2n})$:%
\[
u\in\mathcal{S}^{\prime}(\mathbb{R}^{n})\text{ \textit{and} }\widehat{A}%
u\in\mathcal{S}(\mathbb{R}^{n})\text{ \textit{implies} }u\in\mathcal{S}%
(\mathbb{R}^{n})
\]
(global hypoellipticity is thus a stronger property than that of the usual
hypoellipticity, familiar from the (micro)local analysis of pseudodifferential operators).

We will also need the following elementary result that says that the symbol
classes $H\Gamma_{\rho}^{m_{1},m_{0}}(\mathbb{R}^{2n})$ are invariant under
linear changes of variables:

\begin{lemma}
\label{lemmas}Let $a\in H\Gamma_{\rho}^{m_{1},m_{0}}(\mathbb{R}^{2n})$ with
$m_{0}>0$. For every linear automorphism $f$ of $\mathbb{R}^{2n}$ we have
$f^{\ast}a=a\circ f\in H\Gamma_{\rho}^{m_{1},m_{0}}(\mathbb{R}^{2n})$.
\end{lemma}

\begin{proof}
Set $a^{\prime}(z)=a(fz)$; clearly $a^{\prime}\in C^{\infty}(\mathbb{R}^{2n}%
)$. We now note that there exist $\lambda,\mu>0$ such that $\lambda
|z|\leq|fz|\leq\mu|z|$ for all $z\in\mathbb{R}^{n}$. Since $m_{0}>0$ it
follows that
\[
C_{0}^{\prime}|z|^{m_{0}}\leq|a^{\prime}(z)|\leq C_{1}^{^{\prime}}|z|^{m_{1}}%
\]
with $C_{0}^{\prime}=C_{0}\lambda^{m_{0}}$ and $C_{1}^{^{\prime}}=C_{1}%
\mu^{m_{1}}$. Next, we observe that for every $\alpha\in\mathbb{N}^{2n}$,
$|\alpha|\neq0$, there exists $B_{\alpha}>0$ such that $|\partial_{z}^{\alpha
}a^{\prime}(z)|\leq B_{\alpha}|\partial_{z}^{\alpha}a(fz)|$ (this is easily
seen by induction on $|\alpha|$ and using the chain rule); we thus have%
\[
|\partial_{z}^{\alpha}a^{\prime}(z)|\leq C_{\alpha}B_{\alpha}|a^{\prime
}(z)||fz|^{-\rho|\alpha|}\leq C_{\alpha}^{\prime}|a^{\prime}(z)||z|^{-\rho
|\alpha|}%
\]
with $C_{\alpha}^{\prime}=B_{\alpha}C_{\alpha}\mu^{-\rho|\alpha|}$. Hence
$a^{\prime}\in H\Gamma_{\rho}^{m_{1},m_{0}}(\mathbb{R}^{2n})$.
\end{proof}

\subsection{Application to the operators $\protect\widetilde{A}_{\omega}$}

Let us now apply the theory of Shubin classes to the study of some spectral
properties of the operators $\widetilde{A}_{\omega}$. We begin by studying the
standard case $\Omega=J$; as previously we set $\widetilde{A}_{\omega
}=\widetilde{A}$. The extension to the general case will be done using again
the reduction result in Proposition \ref{proptwo}.

Proposition \ref{propob} is the key to the following general spectral result,
which shows how to obtain the eigenvalues and eigenvectors of $\widetilde{A}$
from those of $\widehat{A}$:

\begin{proposition}
\label{eigenva1}Let $a\in H\Gamma_{\rho}^{m_{1},m_{0}}(\mathbb{R}^{2n})$ be
real, and $m_{0}>0$. Then:

(i) The eigenvalues of the operators $\widehat{A}$ and $\widetilde{A}$ are the
same; and $\widetilde{A}$ has discrete spectrum $(\lambda_{j})_{j\in
\mathbb{N}}$ with $\lim_{j\rightarrow\infty}|\lambda_{j}|=\infty$;

(ii) The eigenfunctions of $\widetilde{A}$ are given by $\Phi_{j,k}%
=W_{\phi_{j}}\phi_{k}$ where the $\phi_{j}$ are the eigenfunctions of the
operator $\widehat{A}$.

(iii) Conversely, if $U$ is an eigenfunction of $\widetilde{A}$, then
$u=W_{\phi}^{\ast}U$ is an eigenvector of $\widehat{A}$ corresponding to the
same eigenvalue.
\end{proposition}

\begin{proof}
That every eigenvalue of $\widehat{A}$ also is an eigenvalue of $\widetilde{A}%
$ is clear: if $\widehat{A}u=\lambda u$ for some $u\neq0$, then
\[
\widetilde{A}(W_{\phi}u)=W_{\phi}\widehat{A}u=\lambda W_{\phi}u
\]
and $U=W_{\phi}u\neq0$; this proves at the same time that $W_{\phi}u$ is an
eigenvector of $\widehat{A}$ because $W_{\phi}$ has kernel $\{0\}$. Assume
conversely that $\widetilde{A}U=\lambda U$ for $U\in L^{2}(\mathbb{R}^{2n})$,
$U\neq0$, and $\lambda\in\mathbb{R}$. For every $\phi$ we have
\[
\widehat{A}W_{\phi}^{\ast}U=W_{\phi}^{\ast}\widetilde{A}U=\lambda W_{\phi
}^{\ast}U
\]
hence $\lambda$ is an eigenvalue of $\widehat{A}$ and $u$ an eigenvector if
$u=W_{\phi}^{\ast}U\neq0$. That $\widetilde{A}$ has discrete spectrum
$(\lambda_{j})_{j\in\mathbb{N}}$ with $\lim_{j\rightarrow\infty}|\lambda
_{j}|=\infty$ now follows from Proposition \ref{propeigen}. We have $W_{\phi
}u=W_{\phi}W_{\phi}^{\ast}U=P_{\phi}U$ where $P_{\phi}$ is the orthogonal
projection on the range $\mathcal{H}_{\phi}$ of $W_{\phi}$. Assume that $u=0$;
then $P_{\phi}U=0$ for every $\phi\in\mathcal{S}(\mathbb{R}^{n})$, and hence
$U=0$ in view of Proposition \ref{propob}.
\end{proof}

Let us now consider the general case of operators $\widetilde{A}_{\omega}$.

\begin{proposition}
Let $a\in H\Gamma_{\rho}^{m_{1},m_{0}}(\mathbb{R}^{2n})$ be real, and
$m_{0}>0$. Then:

(i) The operator $\widetilde{A}_{\omega}$ has discrete spectrum $(\lambda
_{j})_{j\in\mathbb{N}}$ with $\lim_{j\rightarrow\infty}|\lambda_{j}|=\infty$.

(ii) The eigenfunctions of $\widetilde{A}_{\omega}$ are the functions
$\Phi_{j}=W_{f,\phi}\phi_{j}$ where the $\phi_{j}$ are the eigenfunctions of
the operator $\widehat{A}^{\prime}$ with Weyl symbol $a^{\prime}=f^{\ast}a$.

(iii) We have $\Phi_{j,k}=W_{f,\phi_{j}}\phi_{k}\in\mathcal{S}(\mathbb{R}%
^{2n})$ and the $\Phi_{j,k}$ form an orthonormal basis of $\mathcal{S}%
(\mathbb{R}^{2n})$.
\end{proposition}

\begin{proof}
Recall that we have shown in Proposition \ref{aprime} that $\widetilde{A}%
_{\omega}W_{f,\phi}=W_{f,\phi}\widehat{A^{\prime}}$ where $\widehat{A^{\prime
}}\overset{\text{Weyl}}{\longleftrightarrow}a\circ f$. In view of Lemma
\ref{lemmas} the Shubin class $H\Gamma_{\rho}^{m_{1},m_{0}}(\mathbb{R}^{2n})$
is preserved by linear changes of variables. The proof of the Proposition now
follows \textit{mutatis mutandis} from that of Proposition \ref{eigenva1}
replacing $\widehat{A}$ with the operator $\widehat{A^{\prime}}$ with Weyl
symbol $a\circ f$ and using the intertwining formula $\widetilde{A}_{\omega
}W_{f,\phi}=W_{f,\phi}\widehat{A^{\prime}}$ together with the fact that
$W_{f,\phi}=M_{f}^{-1}W_{f,\phi}$ where $M_{f}^{-1}$ is a unitary operator.
\end{proof}

\subsection{Gelfand triples and generalized eigenvalues}

Eigenvectors of pseudo-differential operators are not always elements of a
Hilbert space, but of a distribution space. The notion of Gelfand triple (or
rigged Hilbert spaces, as it was called by the physicist Dirac) formalizes
this observation, that we briefly recall here since it provides the natural
setting for the discussion of the spectral properties of our classes of
pseudodifferential operators, e.g. if the symbol is not an element of
$H\Gamma_{\rho}^{m_{1},m_{0}}(\mathbb{R}^{2n})$.

A (Banach) Gelfand triple $(\mathcal{B},\mathcal{H},\mathcal{B}^{\prime})$
consists of a (Banach) Fr\'{e}chet space $\mathcal{B}$ which is continuously
and densely embedded into a Hilbert space ${\mathcal{H}}$, which in turn is
$w^{\ast}$-continuously and densely embedded into the dual (Banach)
Fr\'{e}chet space $\mathcal{B}^{\prime}$. In this definition one identifies
$\mathcal{H}$ with its dual $\mathcal{H}^{\ast}$ and the scalar product on
$\mathcal{H}$ thus extends in a natural way into a pairing between
$\mathcal{B}\subset\mathcal{H}$ and $\mathcal{B}^{\prime}\supset\mathcal{H}$.

The standard example of a Gelfand triple is $(\mathcal{S}(\mathbb{R}%
^{n}),L^{2}(\mathbb{R}^{n}),\mathcal{S}^{\prime}(\mathbb{R}^{n}))$ but there
are many other examples; one of them is $(M_{0}^{1}(\mathbb{R}^{n}%
),L^{2}(\mathbb{R}^{n}),M_{0}^{1}(\mathbb{R}^{n})^{\prime})$ where $M_{0}%
^{1}(\mathbb{R}^{n})$ is the Feichtinger algebra which is a particular
modulation space (see Subsection \ref{subsecmqs} below). The use of this
Gelfand triple not only offers a better description of self-adjoint operators
but it also allows a simplification of many proofs.

Given a Gelfand triple $(\mathcal{B},\mathcal{H},\mathcal{B}^{\prime})$ one
proves that every self-adjoint operator $A:\mathcal{B}\longrightarrow
\mathcal{B}$ has a complete family of generalized eigenvectors $(\psi_{\alpha
})_{\alpha}=\{\psi_{\alpha}\in\mathcal{B}^{\prime}:\alpha\in\mathbb{A}\}$
($\mathbb{A}$ an index set), defined as follows: for every $\alpha
\in\mathbb{A}$ there exists $\lambda_{\alpha}\in\mathbb{C}$ such that%
\[
(\psi_{\alpha},A\phi)=\lambda_{\alpha}(\psi_{\alpha},\phi)\text{ \ for every
}\phi\in\mathcal{B}\text{.}%
\]
Completeness of the family $(\psi_{\alpha})_{\alpha}$ means that there exists
at least one $\psi_{\alpha}$ such that $(\psi_{\alpha},\phi)$ $\neq0$ for
every $\phi\in\mathcal{B} \backslash\left\{  0 \right\}  $. The scalars
$\lambda_{\alpha}$ are called generalized eigenvectors. For more see
\cite{cofelu08,feko98,fe09}.

\begin{proposition}
Let $a$ be a real-valued symbol and choose $(\mathcal{S}(\mathbb{R}^{n}%
),L^{2}(\mathbb{R}^{n}),\mathcal{S}^{\prime}(\mathbb{R}^{n}))$ as Gelfand triple.

(i) The generalized eigenvalues of $\widetilde{A}_{\omega}$ and those of the
Weyl operator $\widehat{A^{\prime}}\overset{\text{Weyl}}{\longleftrightarrow
}a\circ f$ are the same;

(ii) Let $u$ be a generalized eigenvector of $\widehat{A^{\prime}}$:
$\widehat{A^{\prime}}u=\lambda u$. Then $U=W_{f,\phi}u$ satisfies
$\widetilde{A}_{\omega}U=\lambda U$;

(iii) Conversely, if $U$ is a generalized eigenvector of $\widetilde{A}%
_{\omega}$ then $u=W_{\phi}^{\ast}U$ is a generalized eigenvector of
$\widehat{A^{\prime}}$ corresponding to the same generalized eigenvalue.
\end{proposition}

\begin{proof}
The basic idea is that it suffices to establish the results for the test
functions $\mathcal{S}(\mathbb{R}^{n})$. First note that the assumption on $a$
guarantees the self-adjointness of $\widehat{A},\widetilde{A}$ and
$\widetilde{A}_{\omega}$ therefore it yields the existence of generalized
eigenvalues and eigenvectors. The arguments of the preceding question remain
valid in this context if we interpret them in the weak sense.
\end{proof}

\section{Regularity in Modulation Spaces}

The modulation spaces $M_{v}^{p,q}(\mathbb{R}^{n})$ introduced in the 80's by
Feichtinger \cite{fe81,fe06,fe03-1} and developed by Feichtinger and
Gr\"{o}chenig \cite{gr00} are a tool of choice for relating the regularity
properties of the phase space operator $\widetilde{A}_{\omega}$ to those of
the corresponding operator $\widehat{A}$. In addition, the modulation spaces
$M_{v}^{\infty,1}(\mathbb{R}^{n}\oplus\mathbb{R}^{n})$ (which contain as a
particular case the so-called Sj\"{o}strand class) will supply us with symbol
classes defined without any reference to differentiability properties. We
define the modulation spaces in terms of the cross-Wigner transform; in the
standard literature on the topic (especially in time-frequency analysis) they
are defined using a closely related object, the \textquotedblleft windowed
short-time Fourier transform\textquotedblright. Because of the particular form
of the weighting functions we use, it is easy to see that both definitions coincide.

\subsection{The spaces $M_{s}^{q}$\label{subsecmqs}}

Let $s\geq0$ and set $v_{s}(z)=(1+|z|^{2})^{s/2}$. We note that for every
$f\in GL(2n,\mathbb{R})$ there exists a constant $C_{s,f}$ such that%
\begin{equation}
v_{s}(fz)\leq C_{s,f}v_{s}(z). \label{csf}%
\end{equation}
The modulation space $M_{s}^{q}(\mathbb{R}^{n})$ ($q\geq1$) consists of all
distributions $u\in\mathcal{S}^{\prime}(\mathbb{R}^{n})$ such that
$W(u,\phi)\in L_{s}^{q}(\mathbb{R}^{2n})$ for some window $\phi\in
\mathcal{S}(\mathbb{R}^{n})$; here $L_{s}^{q}(\mathbb{R}^{2n})$ consists of
all functions $U$ on $\mathbb{R}^{2n}$ such that $v_{s}U\in L^{q}%
(\mathbb{R}^{2n})$. One shows that this definition is independent of the
choice of window $\phi$ and that if it holds for one $\phi$ in $\mathcal{S}%
(\mathbb{R}^{n})$ then it holds for all. Moreover the formula
\[
||u||_{\phi,M_{s}^{q}}=|||W_{\phi}u|||_{L_{s}^{q}}=\left(  \int_{\mathbb{R}%
^{2n}}|W_{\phi}u(z)|^{q}v_{s}^{q}(z)dz\right)  ^{\frac{1}{q}}%
\]
defines a norm on $M_{s}^{q}(\mathbb{R}^{n})$ and different $\phi$ lead to
equivalent norms. The topology defined by any of these norm endows $M_{s}%
^{q}(\mathbb{R}^{n})$ with a Banach space structure. The spaces $M_{s}^{q}$
increase with the parameter $q$: if $q\leq q^{\prime}$ then $M_{s}%
^{q}(\mathbb{R}^{n})\subset M_{s}^{q^{\prime}}(\mathbb{R}^{n})$. Following
result summarizes the main algebraic properties of $M_{s}^{q}(\mathbb{R}^{n})$:

\begin{proposition}
(i) The modulation spaces $M_{s}^{q}(\mathbb{R}^{n})$ are invariant under the
action of the metaplectic group $\operatorname*{Mp}(2n,\sigma)$: $u\in
M_{s}^{q}(\mathbb{R}^{n})$ if and only $\widehat{S}u\in M_{s}^{q}%
(\mathbb{R}^{n})$ for every $\widehat{S}\in\operatorname*{Mp}(2n,\sigma)$;

(ii) There exists a constant $C>0$ such that for every $z\in\mathbb{R}^{2n}$
we have%
\[
||\widehat{T}(z)u||_{\phi,M_{s}^{q}}\leq Cv_{s}(z)||u||_{\phi,M_{s}^{q}};
\]
in particular $M_{s}^{q}(\mathbb{R}^{n})$ is invariant under the action of the
Heisenberg--Weyl operators;

(iii) Let $f\in GL(n,\mathbb{R})$. We have $u\in M_{s}^{q}(\mathbb{R}^{n})$ if
and only if $f^{\ast}u=f\circ u\in M_{s}^{q}(\mathbb{R}^{n})$.
\end{proposition}

The properties (i)--(ii) above can be stated in more concise form by saying
that the modulation spaces $M_{s}^{q}(\mathbb{R}^{n})$ are invariant under the
action of the inhomogeneous metaplectic group $\operatorname*{IMp}(2n,\sigma)$
(it is the group of unitary operators generated by the elements of
$\operatorname*{Mp}(2n,\sigma)$ together with the Heisenberg--Weyl operators).

In the particular case $s=0$, $q=1$ one obtains the Feichtinger algebra
$S_{0}(\mathbb{R}^{n})=M^{1}(\mathbb{R}^{n})$. It is an algebra for both
pointwise multiplication and convolution. It is the smallest Banach algebra
containing $\mathcal{S}(\mathbb{R}^{n})$ and invariant under the action of the
Heisenberg--Weyl operators (and hence of $\operatorname*{IMp}(2n,\sigma)$),
and we have%
\[
M^{1}(\mathbb{R}^{n})\subset L^{1}(\mathbb{R}^{n})\cap F(L^{1}(\mathbb{R}%
^{n}));
\]
using the Riemann--Lebesgue theorem it follows in particular that%
\[
M^{1}(\mathbb{R}^{n})\subset C^{0}(\mathbb{R}^{n}).
\]

The following easy observation will be used in the forthcoming sections:

\begin{lemma}
We have $u\in M_{s}^{q}(\mathbb{R}^{n})$ if and only if $W_{f,\phi}u\in
L_{s}^{q}(\mathbb{R}^{n})$.
\end{lemma}

\begin{proof}
Since $W_{f,\phi}=M_{f}^{-1}W_{\phi}$ and $W_{\phi}u$ is proportional to
$W(u,\phi)$ it suffices to show that if $U\in L_{s}^{q}(\mathbb{R}^{2n})$ then
$M_{f}^{-1}U\in L_{s}^{q}(\mathbb{R}^{2n})$. In view of definition
(\ref{pushpush}) of $M_{f}U$ we have, using the inequality (\ref{csf}),%
\begin{align*}
\int_{\mathbb{R}^{2n}}|M_{f}^{-1}U(z)|^{q}v_{s}^{q}(z)dz  &  =|\det
f|^{-1/2}\int_{\mathbb{R}^{2n}}|U(f^{-1}z)|^{q}v_{s}^{q}(z)dz\\
&  =|\det f|^{1/2}\int_{\mathbb{R}^{2n}}|U(z)|^{q}v_{s}^{q}(fz)dz\\
&  \leq C\int_{\mathbb{R}^{2n}}|U(z)|^{q}v_{s}^{q}(z)dz
\end{align*}
which proves the assertion.
\end{proof}

The dual Banach space $M_{0}^{1}(\mathbb{R}^{n})^{\prime}$ consists of all
$u\in S^{\prime}(\mathbb{R}^{n})$ such that $W(u,\phi)\in L^{\infty
}(\mathbb{R}^{2n})$ for some (and hence every) window $\phi\in M_{0}%
^{1}(\mathbb{R}^{n})$; the duality bracket is given by the pairing
\begin{equation}
(u,u^{\prime})=\int_{\mathbb{R}^{2n}}W(u,\phi)(z)\overline{W(u^{\prime}%
,\phi)(z)}dz \label{dual151}%
\end{equation}
and the formula
\begin{equation}
||\psi||_{\phi,(M_{0}^{1})^{\prime}}^{\hbar}=\sup_{z\in\mathbb{R}^{2n}}%
|W(\psi,\phi)(z)| \label{dual152}%
\end{equation}
defines a norm on $M_{0}^{1}(\mathbb{R}^{n})^{\prime}$ for which this space is complete.

\subsection{The symbol class $M_{s}^{\infty,1}$}

Let us now introduce a different class of modulation spaces, which contains as
a particular case the Sj\"{o}strand classes, defined by other methods in
Sj\"{o}strand \cite{sj94}; also see the paper \cite{Boulkhemair} by
Boulkhemair. It is interesting to view these modulation spaces as symbol
classes: in contrast to the cases traditionally considered in the literature,
membership of a symbol $a$ in $M_{s}^{\infty,1}(\mathbb{R}^{n}\oplus
\mathbb{R}^{n})$ does not require any smoothness properties of $a$. It turns
out that this point of view allows to recover many classical and difficult
regularity results (for instance then Calder\'{o}n--Vaillancourt theorem) in a
rather simple way; see for instance Gr\"{o}chenig \cite{grojam,gr06bis}. In a
recent paper \cite{goluseful} two of us pointed out the relevance of
Sj\"{o}strand classes for deformation quantization.

As before we set $v_{s}(z)=(1+|z|^{2})^{s/2}$ for $z\in\mathbb{R}^{2n}$. The
modulation space $M_{s}^{\infty,1}(\mathbb{R}^{n}\oplus\mathbb{R}^{n})$
consists of all distributions in $\mathcal{S}^{\prime}(\mathbb{R}^{2n})$
(viewed as pseudo-differential symbols, and hence denoted $a,b,...$) such that%
\begin{equation}
\sup_{z\in\mathbb{R}^{2n}}|W(a,\Phi)(z,\zeta)v_{s}(z)|\in L^{1}(\mathbb{R}%
^{n}\oplus\mathbb{R}^{n}) \label{miffi14}%
\end{equation}
for every $\Phi\in\mathcal{S}(\mathbb{R}^{2n})$. Here $W(a,\Phi)$ is the
cross-Wigner transform of functions (or distributions) defined on
$\mathbb{R}^{n}\oplus\mathbb{R}^{n}$. When $s=0$ the space $M_{0}^{\infty
,1}(\mathbb{R}^{2n})=M^{\infty,1}(\mathbb{R}^{2n})$ is called the
Sj\"{o}strand class. It thus consists of all symbols $a\in\mathcal{S}^{\prime
}(\mathbb{R}^{n}\oplus\mathbb{R}^{n})$ such that%
\[
\sup_{z\in\mathbb{R}^{2n}}|W(a,\Phi)(z,\zeta)|\in L^{1}(\mathbb{R}^{n}%
\oplus\mathbb{R}^{n})
\]
for every $\Phi\in\mathcal{S}(\mathbb{R}^{2n})$, and we have
\begin{equation}
S_{0,0}^{0}(\mathbb{R}^{2n})\subset C_{b}^{2n+1}(\mathbb{R}^{2n})\subset
M^{\infty,1}(\mathbb{R}^{2n}) \label{c2n}%
\end{equation}
where $C_{b}^{2n+1}(\mathbb{R}^{2n})$ is the vector space of all bounded
complex functions on $\mathbb{R}^{2n}$ with continuous and bounded derivatives
up to order $2n+1$ and the symbol class $S_{0,0}^{0}(\mathbb{R}^{2n})$
consists of all infinitely differentiable complex functions $a$ on
$\mathbb{R}^{n}\oplus\mathbb{R}^{n}$ such that $\partial_{z}^{\alpha}a$ is
bounded for all multi-indices $\alpha\in\mathbb{N}^{2n}$.

It is clear that $M_{s}^{\infty,1}(\mathbb{R}^{2n})$ is a complex vector space
for the usual operations. In fact:

\begin{proposition}
We have $\Psi\in M_{s}^{\infty,1}(\mathbb{R}^{n}\oplus\mathbb{R}^{n})$ if and
only if (\ref{miffi14}) holds for one $\Phi\in\mathcal{S}(\mathbb{R}^{n}%
\oplus\mathbb{R}^{n})$, and

(i) The equalities%
\[
||a||_{M_{s}^{\infty,1}}^{\Phi}=\int_{\mathbb{R}^{2n}}\sup_{z\in
\mathbb{R}^{2n}}|W(a,\Phi)(z,\zeta)v_{s}(z)|d\zeta
\]
define a family of equivalent norms on $M_{s}^{\infty,1}(\mathbb{R}^{n}%
\oplus\mathbb{R}^{n})$ for different $\Phi\in\mathcal{S}(\mathbb{R}^{2n})$;

(ii) The space $M_{s}^{\infty,1}(\mathbb{R}^{n}\oplus\mathbb{R}^{n})$ is a
Banach space for the topology defined by any of the norms $||\cdot
||_{M_{s}^{\infty,1}}^{\Phi}$ and $\mathcal{S}(\mathbb{R}^{2n})$ is a dense
subspace of $M_{s}^{\infty,1}(\mathbb{R}^{n}\oplus\mathbb{R}^{n})$.
\end{proposition}

The interest of $M_{s}^{\infty,1}(\mathbb{R}^{n}\oplus\mathbb{R}^{n})$ comes
from the following property of the twisted product (Gr\"{o}chenig
\cite{gr06bis}):

\begin{proposition}
Let $a,b\in M_{s}^{\infty,1}(\mathbb{R}^{2n})$. Then $a\#b\in M_{s}^{\infty
,1}(\mathbb{R}^{n}\oplus\mathbb{R}^{n})$. In particular, for every window
$\Phi$ there exists a constant $C_{\Phi}>0$ such that
\[
||a\#b||_{M_{s}^{\infty,1}}^{\Phi}\leq C_{\Phi}||a||_{M_{s}^{\infty,1}}^{\Phi
}||b||_{M_{s}^{\infty,1}}^{\Phi}.
\]

\end{proposition}

Recall that the twisted product $a\#b$ is the Weyl symbol of the product
$\widehat{A}\widehat{B}$ of the operators $\widehat{A}\overset{\text{Weyl}%
}{\longleftrightarrow}a$ and $\widehat{B}\overset{\text{Weyl}%
}{\longleftrightarrow}b$. Since obviously $\overline{a}\in M_{s}^{\infty
,1}(\mathbb{R}^{n}\oplus\mathbb{R}^{n})$ if and only and $a\in M_{s}%
^{\infty,1}(\mathbb{R}^{n}\oplus\mathbb{R}^{n})$ the property above can be
restated by saying that $M_{s}^{\infty,1}(\mathbb{R}^{2n})$ is a Banach $\ast
$-algebra with respect to the twisted product $\#$ and the involution
$a\longmapsto\overline{a}$.

The following property follows from Theorem 4.1 and its Corollary 4.2 in
\cite{gr06bis} (also see (\cite{gr00}, Theorem 14.5.6); it is a particular
case of more general results in Toft \cite{A8}.

In the case of the Sj\"{o}strand class $M^{\infty,1}(\mathbb{R}^{n}%
\oplus\mathbb{R}^{n})$ one has the following more precise results:

\begin{proposition}
\label{propsjo}Let $\widehat{A}\overset{\text{Weyl}}{\longleftrightarrow}a$.
We have:

(i) If $a\in M^{\infty,1}(\mathbb{R}^{n}\oplus\mathbb{R}^{n})$ then
$\widehat{A}$ is bounded on $L^{2}(\mathbb{R}^{n})$ and on all $M^{q}%
(\mathbb{R}^{n})=M_{0}^{q}(\mathbb{R}^{n})$;

(ii) If $a\in M_{s}^{\infty,1}(\mathbb{R}^{n}\oplus\mathbb{R}^{n})$ then
$\widehat{A}$ is bounded on every modulation space $M_{s}^{q}(\mathbb{R}^{n})$;

(iii) If $\widehat{A}$ with $a\in M^{\infty,1}(\mathbb{R}^{n}\oplus
\mathbb{R}^{n})$ is invertible with inverse $\widehat{B}\overset{\text{Weyl}%
}{\longleftrightarrow}b$ then $b\in M^{\infty,1}(\mathbb{R}^{n}\oplus
\mathbb{R}^{n})$.
\end{proposition}

Property (i) thus extends the $L^{2}$-boundedness property of operators with
symbols in $S_{0,0}^{0}(\mathbb{R}^{n}\oplus\mathbb{R}^{n})$. Property (iii)
is called the \textit{Wiener property} of $M^{\infty,1}(\mathbb{R}^{2n})$.

\subsection{Regularity results}

Before we prove our main result, Proposition \ref{propreg}, let us show that
the symbol spaces $M_{s}^{\infty,1}(\mathbb{R}^{2n})$ are invariant under
linear changes of variables:

\begin{lemma}
\label{lemmaf}Let $f\in GL(2n,\mathbb{R})$ and set $f^{\ast}a=a\circ f$. There
exists a constant $C_{A}>0$ such that
\begin{equation}
||f^{\ast}a||_{\Phi,M_{s}^{\infty,1}}\leq C_{s}||a||_{(f^{-1})^{\ast}%
\Phi,M_{s}^{\infty,1}} \label{CA14}%
\end{equation}
for every $\Phi\in\mathcal{S}(\mathbb{R}^{n}\oplus\mathbb{R}^{n})$. In
particular $a\in M_{s}^{\infty,1}(\mathbb{R}^{2n})$ if and only $f^{\ast}a\in
M_{s}^{\infty,1}(\mathbb{R}^{n}\oplus\mathbb{R}^{n})$.
\end{lemma}

\begin{proof}
Let us set $b=f^{\ast}a$. We have, by definition of the cross-Wigner
transform,
\[
W(b,\Phi)(z,\zeta)=\left(  \tfrac{1}{2\pi}\right)  ^{2n}\int_{\mathbb{R}^{2n}%
}e^{-i\zeta\cdot\eta}a(fz+\tfrac{1}{2}f\eta)\overline{\Phi(z-\tfrac{1}{2}%
\eta)}d\eta
\]
thus, performing the change of variables $\xi=f\eta$,%
\begin{multline*}
W(b,\Phi)(f^{-1}z,f^{T}\zeta)=\left(  \tfrac{1}{2\pi}\right)  ^{2n}|\det
f|^{-1}\\
\times\int_{\mathbb{R}^{2n}}e^{-i\zeta\cdot\xi}a(z+\tfrac{1}{2}\xi
)\overline{(f^{-1})^{\ast}\Phi(z-\tfrac{1}{2}\xi)}d\xi.
\end{multline*}
and hence%
\begin{equation}
W(b,\Phi)(z,\zeta)=|\det f|^{-1}W(a,(f^{-1})^{\ast}\Phi)(fz,(f^{T})^{-1}%
\zeta);\label{proofmp14}%
\end{equation}
taking the suprema of both sides of this equality and integrating we get%
\[
||f^{\ast}a||_{M_{s}^{\infty,1}}^{\Phi}=\int_{\mathbb{R}^{2n}}\sup
_{z\in\mathbb{R}^{2n}}|W(a,(f^{-1})^{\ast}\Phi)(z,\zeta)v_{s}(f^{-1}z)|d\zeta
\]
Since $v_{s}(f^{-1}z)\leq C_{s,f}v_{s}(z)$ for some constant $C_{s,f}>0$ (cf.
the inequality (\ref{csf})) the estimate (\ref{CA14}) follows.
\end{proof}

Let us now introduce the following notation: for an arbitrary window $\phi$
set
\begin{equation}
\mathcal{L}_{f,\phi}^{q}(\mathbb{R}^{2n})=W_{f,\phi}(M_{s}^{q}(\mathbb{R}%
^{n}))\subset L_{s}^{q}(\mathbb{R}^{2n}). \label{lqff}%
\end{equation}
Clearly $\mathcal{L}_{f,\phi}(\mathbb{R}^{2n})$ is a closed linear subspace of
$L_{s}^{q}(\mathbb{R}^{2n})$.

\begin{proposition}
\label{propreg}Let $\widetilde{A}_{\omega}$ be associated to the Weyl operator
$\widehat{A}\overset{\text{Weyl}}{\longleftrightarrow}a$. If $a\in
M_{s}^{\infty,1}(\mathbb{R}^{2n})$ then
\[
\widetilde{A}_{\omega}:\mathcal{L}_{f,\phi}^{q}(\mathbb{R}^{2n}%
)\longrightarrow\mathcal{L}_{f,\phi}^{q}(\mathbb{R}^{2n})
\]
(continuously) for every window $\phi\in\mathcal{S}(\mathbb{R}^{n})$.
\end{proposition}

\begin{proof}
Let $U\in\mathcal{L}_{f,\phi}^{q}(\mathbb{R}^{2n})$; by definition there
exists $u\in M_{s}^{q}(\mathbb{R}^{n})$ such that $U=W_{f,\phi}u$. In view of
the first intertwining relation (\ref{sfund}) we have
\[
\widetilde{A}_{\omega}W_{f,\phi}u=W_{f,\phi}\widehat{A^{\prime}}u
\]
where $\widehat{A^{\prime}}\overset{\text{Weyl}}{\longleftrightarrow}%
a^{\prime}$ with $a^{\prime}(z)=a(fz)$. In view of Lemma \ref{lemmaf} above we
have $a^{\prime}\in M_{s}^{\infty,1}(\mathbb{R}^{2n})$ and hence
$\widehat{A^{\prime}}u\in M_{s}^{q}(\mathbb{R}^{n})$ and is bounded in view of
Proposition \ref{propsjo}(ii). It follows that $W_{f,\phi}\widehat{A^{\prime}%
}u\in\mathcal{L}_{f,\phi}^{q}(\mathbb{R}^{2n})$.
\end{proof}

It is worthwhile (and important, in a quantum mechanical context) to note that
the spaces $\mathcal{L}_{f,\phi}^{q}(\mathbb{R}^{2n})$ cannot contain
functions which are \textquotedblleft too concentrated\textquotedblright%
\ around a point; this is reminiscent of the uncertainty principle. In
particular the Schwartz space $\mathcal{S}(\mathbb{R}^{2n})$ is not contained
in any of the $\mathcal{L}_{f,\phi}^{q}(\mathbb{R}^{2n})$. This observation is
based on the following result, proved in de Gosson and Luef
\cite{GOLULETT,GOLUAHP} using Hardy's uncertainty principle for a function and
its Fourier transform: assume that $u\in\mathcal{S}(\mathbb{R}^{n})$ is such
that $Wu\leq Ce^{-Mz\cdot z}$ for some $C>0$ and a real matrix $M=M^{T}>0$.
Consider now the eigenvalues of $JM$; these are of the form $\pm i\lambda_{j}$
with $\lambda_{j}>0$. Then we must have $\lambda_{j}\leq1$ for all
$j=1,...,n$. Equivalently , the symplectic capacity $c(\mathcal{W}_{M})$ of
the \textquotedblleft Wigner ellipsoid\textquotedblright\ $\mathcal{W}%
_{M}:Mz\cdot z\leq1$ satisfies $c(\mathcal{W})\geq\pi$. [Recall
\cite{HZ,Polter} that the symplectic capacity of an ellipsoid $\mathcal{W}$ in
$\mathbb{R}^{2n}$ is the number $\pi R^{2}$ where $R$ is the supremum of the
radii of all balls $B^{2n}(r)$ that can be sent into $\mathcal{W}_{M}$ using
symplectomorphisms of $(\mathbb{R}^{2n},\sigma)$]. This result in fact also
holds true for the cross-Wigner transform \cite{grzi00}: if $|W(u,\phi
)(z)|\leq Ce^{-Mz\cdot z}$ for some $\phi\in\mathcal{S}(\mathbb{R}^{n})$ then
$c(\mathcal{W})\geq\pi$. Assume now that $U\in\mathcal{L}_{f,\phi}%
^{q}(\mathbb{R}^{2n})$ satisfies the sub-Gaussian estimate $|U(z)|\leq
Ce^{-Mz\cdot z}$; by definition of $\mathcal{L}_{f,\phi}^{q}(\mathbb{R}^{2n})$
this is equivalent to%
\[
|W(u,\phi)(fz)|\leq Ce^{-(f^{-1})^{T}Mf^{-1}z\cdot z}%
\]
hence the ellipsoid $f(\mathcal{W}_{M})$ must have symplectic capacity at
least equal to $\pi$. We remark that a complete characterization of the spaces
$M_{s}^{q}(\mathbb{R}^{n})$ and $\mathcal{L}_{f,\phi}^{q}(\mathbb{R}^{2n})$ in
terms of the uncertainty principle is still lacking; we hope to come back to
this important question in a near future.

We finally notice that Lieb \cite{Lieb} has studied integral bounds for
ambiguity and Wigner distributions; how are his results related to ours? This
is certainly worth being explored, especially since he obtains an interesting
characterization for Gaussians in terms of $L^{2}$ norms. In \cite{Bonami}
Bonami et al. extend Beurling's uncertainty principle into a characterization
of Hermite functions. They obtain sharp results for estimates of the Wigner
distribution; it would perhaps be useful to study their results in our
context; we hope to come back to these possibilities in a near future.

\begin{acknowledgement}
Maurice de Gosson has been financed by the Austrian Research Agency FWF
(Projektnummer P20442-N13). Nuno Costa Dias and Jo\~{a}o Nuno Prata have been
supported by the grants PDTC/MAT/ 69635/2006 and PTDC/MAT/099880/2008 of the
Portuguese Science Foundation (FCT). Franz Luef has been financed by the Marie
Curie Outgoing Fellowship PIOF 220464.
\end{acknowledgement}

\begin{acknowledgement}
The authors would like to express their gratitude to the referee for useful
and constructive comments.
\end{acknowledgement}

\[
\]

\textbf{Author's addresses:}%
\[
\]

\textbf{Nuno Costa Dias and Jo\~{a}o Nuno Prata}

\textit{Departamento de Matem\'{a}tica. Universidade Lus\'{o}fona de
Humanidades }

\textit{e Tecnologias. Av. Campo Grande, 376, }

\textit{1749-024 Lisboa, Portugal}

\textit{and}

\textit{Grupo de F\'{\i}sica Matem\'{a}tica, }

\textit{Universidade de Lisboa, }

\textit{Av. Prof. Gama Pinto 2, }

\textit{1649-003 Lisboa, Portugal}\bigskip

\textbf{Maurice de Gosson and Franz Luef}

\textit{Universit\"{a}t Wien, NuHAG}

\textit{Fakult\"{a}t f\"{u}r Mathematik }

\textit{Wien 1090, Austria} \bigskip

\textbf{Franz Luef}

%\textit{Current Address}

\textit{Department of Mathematics}

\textit{UC Berkeley}

\textit{847 Evans Hall}

\textit{Berkeley, CA 94720-3840, USA}

\end{document}